\documentclass{amsart}
\usepackage{amssymb}
\usepackage{amsfonts}
\usepackage{amscd}
\usepackage{amsmath}
\usepackage{mathrsfs}
\usepackage[shortlabels]{enumitem}
\usepackage{epsfig}
\usepackage{graphicx}
\usepackage{verbatim}
\usepackage{latexsym}
\usepackage{eucal}
\usepackage{bbm}

\usepackage{mathtools, array, hyperref}

\def \mbr {{\mathbb R}}

\def \beqq {\begin{equation}}
\def \eeqq {\end{equation}}

\def \dim {\text{dim}}
\def \bpf {\begin{proof}}
\def \epf {\end{proof}}
\def \beq {\begin{equation*}}
\def \eeq {\end{equation*}}

\def \lap {\Delta}
\def \p {\partial}

\newtheorem{theorem}{Theorem}[section]
\theoremstyle{plain}
\newtheorem{corollary}[theorem]{Corollary}
\newtheorem{lemma}[theorem]{Lemma}
\newtheorem{proposition}[theorem]{Proposition}
\theoremstyle{remark}

\numberwithin{equation}{section}

\newcommand{\tr}{\operatorname{tr}}

\newcommand{\vol}{\operatorname{Vol}}

\newcommand{\re}{\operatorname{Re}}
\newcommand{\im}{\operatorname{Im}}
\newcommand{\res}{\operatorname{Res}}
\newcommand{\rank}{\operatorname{rank}}

\newcommand{\supp}{\operatorname{supp}}

\newcommand{\bbR}{\mathbb{R}}
\newcommand{\bbH}{\mathbb{H}}
\newcommand{\bbC}{\mathbb{C}}
\newcommand{\bbS}{\mathbb{S}}

\newcommand{\bbN}{\mathbb{N}}
\newcommand{\calR}{\mathcal{R}}
\newcommand{\calS}{\mathcal{S}}

\newcommand{\cinf}{C^\infty}
\newcommand{\del}{\partial}

\newcommand{\bQ}{\mathbf{Q}}
\newcommand{\vep}{\varepsilon}

\newcommand{\nh}{\tfrac{n}{2}}
\newcommand{\Thetasc}{\Theta_{\rm sc}}

\DeclarePairedDelimiter\paren{(}{)}
\DeclarePairedDelimiter\set{\{}{\}}
\DeclarePairedDelimiter\sqbrak{[}{]}
\DeclarePairedDelimiter\abs{|}{|}
\DeclarePairedDelimiter\norm{\Vert}{\Vert}
\DeclarePairedDelimiter\brak{\langle}{\rangle}

\newcommand{\Hn}{\bbH^{n+1}}

\begin{document}
\title[Existence of resonances]{Existence of resonances for Schr\"odinger operators on hyperbolic space}
\author[Borthwick]{David Borthwick}
\address{Department of Mathematics, Emory University, Atlanta, GA 30322}
\email{dborthw@emory.edu}
\author[Wang]{Yiran Wang}
\address{Department of Mathematics, Emory University, Atlanta, GA 30322}
\email{yiran.wang@emory.edu}
\date{\today}

\begin{abstract}
We prove existence results and lower bounds for the resonances of Schr\"odinger operators 
associated to smooth, compactly support potentials on hyperbolic space.  The results are derived 
from a combination of heat and wave trace expansions and asymptotics of the scattering phase.
\end{abstract}

\maketitle

\section{Introduction}
 
This article is devoted to establishing lower bounds on the resonance count
for Schr\"odinger operators on the hyperbolic space $\bbH^{n+1}$.  Although such results are
well known in Euclidean scattering theory,
the literature for Schr\"odinger operators on hyperbolic space is comparatively
sparse.  Upper bounds on resonances for such operators were considered in \cite{Borthwick:2010, BC:2014}.
Most other recent papers dealing with Schr\"odinger operators on hyperbolic space 
have focused on applications to nonlinear Sch\"rodinger equations 
\cite{AV:2009,Banica:2007,BCD:2009,BCS:2008,BM:2015,Ionescu:2012,IS:2009}.
As far as we are aware, the literature contains no general existence results for resonances in this context.

Let $\lap$ denote the positive Laplacian operator on $\Hn$.  For 
$V \in \cinf_0(\Hn, \bbR)$, the Schr\"odinger operator $\lap+V$
has continuous spectrum $[n^2/4,\infty)$.  
The resolvent of $\lap +V$ is defined for $\re s > n/2$ by 
\begin{equation}\label{RV.def}
R_V(s) := (\lap + V - s(n-s))^{-1}.
\end{equation}
As an operator on weighted $L^2$ spaces, $R_V(s)$ admits a meromorphic extension to $s \in \bbC$,
with poles of finite rank, as described in \S\ref{setup.sec}.
The resonance set $\calR_V$ associated to $V$ consists of the poles of $R_V(s)$, repeated according to the
multiplicity given by
\begin{equation}\label{mv.def}
m_V(\zeta) := \rank \res_\zeta R_V(s).
\end{equation}

There are no eigenvalues embedded in the continuous spectrum, and no resonances on the line $\re s = n/2$ except possibly
at $n/2$.  For a proof, see for example Borthwick-Marzuola \cite[Thm.~1]{BM:2015}. 
The discrete spectrum of $\lap + V$ is therefore finite and lies below $n^2/4$.  
An eigenvalue $\lambda$ corresponds to a resonance $\zeta \in (n/2,n)$ 
such that $\lambda = \zeta(n-\zeta)$. 

The resonance counting function,
 \begin{equation}\label{NV.def}
 N_V(r) := \#\set*{ \zeta \in \calR_V: \abs{\zeta - \tfrac{n}2} \le r},
 \end{equation}
 satisfies a polynomial bound as $r \to \infty$ 
 \begin{equation}\label{NV.bound}
 N_V(r)  = O(r^{n+1}).
 \end{equation}
This estimate is covered by the more general result of Borthwick \cite[Thm.~1.1]{Borthwick:2008}, but with $\Hn$
as the background metric one could also give a simpler proof following the approach that
Guillop\'e-Zworski \cite{GZ:1995b} used for $n=1$.  
A sharp constant for the bound \eqref{NV.bound}, depending on the support of the potential, was obtained in 
Borthwick \cite{Borthwick:2010}.  The corresponding lower bound was shown to hold in a generic sense in
Borthwick-Crompton \cite{BC:2014}, but the existence question was not resolved.  

The existence problem for resonances looks quite different in even and odd dimensions.  In even dimensions, $\calR_0$ contains
resonances at negative integers, with multiplicities such that the polynomial bound \eqref{NV.bound}
is already saturated for $V=0$.  Therefore our goal in even dimensions is to distinguish
$\calR_V$ from $\calR_0$.  On the other hand, $\calR_0$ is empty for odd dimensional hyperbolic space.
In that case we seek lower bounds on $\calR_V$ itself.  

In the present paper, we will prove the following:
\begin{theorem}\label{main.thm}
Let $\calR_V$ denote the set of resonances of $\lap + V$ for  $V \in \cinf_0(\Hn,\bbR)$, with $N_V(r)$ the corresponding counting function.
\begin{enumerate}[label=(\roman*), parsep=4pt, leftmargin=2\parindent]
\item  If the dimension is even or equal to 3, then $\calR_V = \calR_0$ only if $V=0$.   
\item For even dimension $\ge 6$, if $V \ne 0$ then $\calR_V$ and $\calR_0$
differ by infinitely many points.  This conclusion holds also if $\int V\>dg \ge 0$ for $\dim=2$, and  if $\int V\>dg \ne 0$ for $\dim =4$.
\item  In all odd dimensions, if $\calR_V$ is not empty then the resonance set is infinite and $N_V(r) \ne O(r)$.
\end{enumerate}
\end{theorem}

In even dimensions, we also show that the resonance set determines the scattering phase and wave trace completely, 
and in particular fixes all of the wave invariants.  See \S\ref{exist.sec} for the full set of inverse scattering results.

Theorem~\ref{main.thm} is derived from the asymptotic expansions of the scattering phase and heat and wave traces.
In the context of potential scattering in hyperbolic space, these expansions do not seem to have not been 
studied in the literature, so we give a full account of their adaptation to this setting.  The explicit formulas for the
wave invariants are stated in Proposition~\ref{d12.prop}.

The organization of the paper is as follows.  After reviewing some facts on the resolvent and its kernel in \S\ref{setup.sec},
we use the spectral resolution to define distributional traces in \S\ref{trace.sec}.
In \S\ref{bk.sec} we establish the Birman-Krein formula relating these traces to the scattering phase.
The Poisson formula expressing the wave trace as a sum over resonances is proven in \S\ref{poisson.sec}.
In \S\ref{wave.exp.sec} the asymptotic expansion of the wave trace at $t=0$ is established.  The corresponding heat
trace expansion is worked out in \S\ref{heat.sec}, which is then used to study asymptotics of the scattering
phase in \S\ref{scphase.sec}.  Finally, in \S\ref{exist.sec} these tools are applied to derive the existence results.

\medskip

\noindent
{\bf Acknowledgement} The authors are grateful for some very helpful comments and corrections by the journal referees.

\section{The resolvent}\label{setup.sec}

The resolvent of the free Laplacian on $\Hn$ is traditionally written with spectral parameter $s(n-s)$ as in \eqref{RV.def}.
The resolvent kernel is given by a well-known hypergeometric formula, derived by Patterson \cite[Prop~2.2]{Patterson:1989}:
\[
\begin{split}
R_0(s;z,w) &=  \frac{\pi^{-\frac{n}2} 2^{-2s-1}\Gamma(s)}{\Gamma(s-\frac{n}2+1)} \cosh^{-2s}(\tfrac12 d(z,w)) \\
&\qquad \times F(s,s-\tfrac{n-1}2, 2s-n+1; \cosh^{-2}(\tfrac12 d(z,w)),
\end{split}
\]
where $d(z,w)$ is the hyperbolic distance.
Using hypergeometric identities \cite[\S14.3(iii)]{NIST}, we can rewrite this formula as
\begin{equation}\label{R0.kernel}
R_0(s;z,w) = (2\pi)^{-\frac{n+1}2} \frac{\Gamma(s)}{{\sinh^\mu d(z,w)}} \bQ_\nu^\mu(\cosh d(z,w)),
\end{equation}
where $\nu := s - \tfrac{n+1}2$, $\mu := \frac{n-1}2$, and $\bQ_\nu^\mu$
denotes the normalized Legendre function,
\[
\bQ_\nu^\mu(x) := \frac{e^{-i \pi \mu}}{\Gamma(\mu+\nu+1)} Q_\nu^\mu(x).
\]
(Under this convention  $\bQ_\nu^\mu(x)$ is entire as a function of both indices.)
The factor $\Gamma(s)$ in \eqref{R0.kernel} has poles at negative integers, but these yield 
resonances only for $n+1$ even.  In odd dimensions the poles are canceled by zeroes of $\bQ_\nu^\mu$. 

Let $(r,\omega) \in [0,\infty) \times \bbS^n$ denote geodesic polar coordinates on $\Hn$.  We will take
\[
\rho := \frac{1}{\cosh r}
\]
as a boundary defining function for the radial compactification of $\Hn$ into a ball. 
The hypergeometric formula for $\bQ_\nu^\mu(x)$ \cite[\S3.2(5)]{Erdelyi} yields an expansion
of the resolvent kernel,
\begin{equation}\label{R0.expand}
R_0(s;z,w) = \pi^{-\frac{n}2} 2^{-s-1} \frac{\Gamma(s)}{\Gamma(s-\frac{n}2+1)} \sum_{k=0}^\infty \frac{a_k(s)}{(\cosh d(z,w))^{s+2k}},
\end{equation}
with $a_0(s) = 1$. In particular,
\begin{equation}\label{R0.infty}
R_0(s;z,w) = O(e^{-sd(z,w)})\text{ as }d(z,w) \to \infty,
\end{equation}
which shows that $R_0(s)$ extends meromorphically to $s \in \bbC$, as a bounded operator 
$\rho^N L^2(\Hn) \to \rho^{-N} L^2(\Hn)$ for $\re s > -N+\nh$.

For $V \in \cinf_0(\Hn, \bbR)$, the resolvent $R_V(s)$ defined in \eqref{RV.def} is related to $R_0(s)$ by the identity,
\begin{equation}\label{R0RV}
R_0(s) = R_V(s) (1 + VR_0(s)).
\end{equation}
The operator $1 + VR_0(s)$ is invertible by Neumann series for $\re s$ sufficiently large, and it follows 
from \eqref{R0.infty} that the operator $VR_0(s)$ is compact on 
$\rho^N L^2(\Hn)$ for $\re s > -N+\nh$.
Hence the analytic Fredholm theorem yields a meromorphic inverse $(1 + VR_0(s))^{-1}$, with poles of finite rank,
which is bounded on $\rho^N L^2(\Hn)$ for $\re s > -N+\nh$.
We thus obtain a meromorphic extension of $R_V(s)$ by setting
\[
R_V(s) = R_0(s) (1 + VR_0(s))^{-1},
\]
which is bounded as an operator $\rho^N L^2(\Hn) \to \rho^{-N} L^2(\Hn)$ for $\re s > -N+\nh$.

We can see from \eqref{R0.expand} that the free resolvent kernel $R_0(s;z,z')$ is 
polyhomogeneous as a function of $\rho(z')$ as $\rho(z') \to 0$, with leading term of order $\rho(z')^s$.
It follows from \eqref{R0RV} that the kernel of $R_V(s)$ has the same property.  
The Poisson kernel is defined as the leading coefficient in the expansion as $\rho(z') \to 0$, 
\begin{equation}\label{EV.def}
E_V(s;z,\omega')  := \lim_{r' \to \infty} \rho(z')^{-s} R_V(s; z,z'),
\end{equation}
Interpreting this function is as an integral kernel, with respect to the standard sphere metric, 
defines the Poisson operator,
\[
E_V(s): \cinf(\bbS^n) \to L^2(\Hn).
\] 
The Poisson operator maps
boundary data to solutions of the generalized eigenfunction equation $(\lap + V - s(n-s))u = 0$.

By Stone's formula, the continuous part of the spectral resolution of $\lap + V$ is given by the restriction of the operator 
$R_V(s) - R_V(n-s)$ to the line $\re s = n/2$.  This is related to the Poisson operator by the following identity:
\begin{equation}\label{RR.EE}
R_V(s) - R_V(n-s) = (n-2s) E_V(s) E_V(n-s)^t,
\end{equation}
as operators $\cinf_0(\Hn) \to L^2(\Hn)$, meromorphically for $s \in \bbC$.
The proof of \eqref{RR.EE} is essentially the same as in the case $n=1$ presented in \cite[Prop.~4.6]{B:STHS2}.

\section{Traces}\label{trace.sec}

Given $V \in \cinf_0(\Hn, \bbR)$ and $f \in \calS(\bbR)$, the operator $f(\lap + V) - f(\lap)$ is of trace class.  
In fact, the map 
\begin{equation}\label{rel.trace}
f \mapsto \tr \sqbrak*{f(\lap + V) - f(\lap)}.
\end{equation}
defines a tempered distribution.   For the proof, see Dyatlov-Zworski \cite[Thm.~3.50]{DZbook}, which applies to
the hyperbolic setting with only minor modifications.

The spectral theorem gives the representation
\[
f(\lap + V) =  \lim_{\vep\to 0}  \frac{1}{2\pi i} \int_{-\infty}^\infty 
\sqbrak[\Big]{(\lap + V - \lambda- i\vep)^{-1} -  (\lap + V - \lambda+i\vep)^{-1}} f(\lambda) d\lambda,
\]
with the limit taken in the operator-norm topology.   
We can separate the contributions from the discrete and continuous spectrum,
and write the continuous part in terms of $R_V(s)$ by setting $s(n-s) = \lambda \pm i \vep$.  The result is
\begin{equation}\label{fDV.tlim}
\begin{split}
f(\lap + V) &=  \lim_{\vep \to 0}  \frac{1}{2\pi i} \int_{-\infty}^\infty 
\sqbrak[\Big]{R_V(\tfrac{n}2 - i \xi + \vep) - R_V(\tfrac{n}2 + i \xi + \vep)} f(\tfrac{n^2}4 + \xi^2) \xi\>d\xi \\
&\qquad + \sum_{j=1}^d f(\lambda_j) \phi_j \otimes \overline{\phi}_j,
\end{split}
\end{equation}
where the $\lambda_1,\dots, \lambda_d$ are the eigenvalues of $\lap+V$, with corresponding normalized eigenvectors $\phi_j$.

The self-adjointness of $\lap + V$ implies an estimate
\[
\norm{(\lap + V - s(n-s)) u} \ge \abs{\im s(n-s)} \norm{u}^2.
\]
This shows that a pole of $R_V(s)$ at $s = \tfrac{n}2$ could have at most second order.  
(This argument is analogous to the Euclidean case; see \cite[Lemma~3.16]{DZbook}.)
A pole of  order two can occur only if $n^2/4$
is an eigenvalue, which is ruled out by Bouclet \cite[Cor.~1.2]{Bouclet:2013}.  Therefore $R_V(s)$ has at most a first order pole at $n/2$.
The integrand in \eqref{fDV.tlim} is thus continuous at $\vep=0$ in the operator topology, because a pole would be canceled by 
the extra factor of $\xi$.  
Taking the limit $\vep \to 0$ gives
\[
\begin{split}
f(\lap + V) &=  \frac{1}{2\pi i} \int_{-\infty}^\infty \sqbrak[\Big]{R_V(\tfrac{n}2 - i \xi) - R_V(\tfrac{n}2 + i \xi)} 
f(\tfrac{n^2}4 + \xi^2) \xi\>d\xi \\
&\qquad + \sum_{j=1}^d f(\lambda_j) \phi_j \otimes \overline{\phi}_j,
\end{split}
\]

Let us define the integral kernel of the spectral resolution as 
\begin{equation}\label{KV.def}
K_V(\xi; z,w) := \frac{\xi}{2\pi i} \sqbrak[\Big]{R_V(\tfrac{n}2 - i \xi; z,w) - R_V(\tfrac{n}2 + i \xi; z,w)}. 
\end{equation}
For $V=0$ this kernel can be written explicitly, using \eqref{R0.kernel} and the Legendre connection formula \cite[\S14.9(iii)]{NIST},
\[
\frac{\bQ_{-\nu-1}^\mu(x)}{\Gamma(\mu+\nu+1)} - \frac{\bQ_{\nu}^\mu(x)}{\Gamma(\mu-\nu) } = \cos(\pi \nu) P_{\nu}^{-\mu}(x).
\]
The result is
\begin{equation}\label{imR0}
K_0(\xi;z,w) :=  c_n(\xi) (\sinh r)^{-\mu} P_{-\frac12 + i\xi}^{-\mu}(\cosh r),
\end{equation}
where $\mu = \tfrac{n-1}2$, $r = d(z,w)$, and
\[
c_n(\xi) := (2\pi)^{-\mu} \xi \sinh (\pi \xi) \Gamma(\tfrac{n}2 + i\xi) \Gamma(\tfrac{n}2 - i\xi).
\]
The hypergeometric expansion \cite[eq.~(14.3.9)]{NIST} of $P_\nu^{-\mu}(x)$ near $x=1$ shows that 
$K_0(\xi;z,w)$ is smooth for all $z,w \in \Hn\times \Hn$. 

For the Schr\"odinger operator case, we note that \eqref{R0RV} yields the identity
\[
R_V(s) - R_V(n-s) = (1 - R_V(s)V)(R_0(s) - R_0(n-s))(1 - VR_V(n-s)).
\]
Since $R_0(s) - R_0(n-s)$ has a smooth kernel for $\re s = n/2$, $V\in \cinf_0(\Hn, \bbR)$, and 
$R_V(s)$ is a pseudodifferential operator of order $-2$, the identity implies that $R_V(s) - R_V(n-s)$ also
has a smooth kernel for $\re s = n/2$, $s \ne n/2$.  The kernel
$K_V(\xi;\cdot,\cdot)$ is thus continuous for $\xi \in \bbR$ and smooth as a function on $\Hn\times \Hn$.

In Borthwick-Marzuola \cite[Prop.~6.1]{BM:2015}, it was shown that $\im R_V(\nh + i\xi;z,w)$
satisfies a polynomial bound as a function of $\xi \in \bbR$, uniformly in $\Hn\times \Hn$,
provided there is no resonance at $s = \nh$.  This restriction can be removed for the $K_V$ estimate 
because of the extra factor of $\xi$ in \eqref{KV.def}, since, as noted above, $R_V(s)$ has at most a first order pole 
at $s= \nh$. We can thus use the spectral resolution formula to write the kernel of $f(\lap + V)$ as
\begin{equation}\label{fv.stone}
f(\lap + V)(z,w) = \int_{-\infty}^\infty K_V(\xi;z,w) f(\tfrac{n^2}4 + \xi^2)\> d\xi 
+ \sum_{j=1}^d f(\lambda_j) \phi_j(z) \bar\phi_j(w).
\end{equation}
Since $f(\lap + V) - f(\lap)$ is trace-class and has a continuous kernel, 
the trace can be computed as an integral over the kernel by
Duflo's theorem \cite[Thm.~V.3.1.1]{Duflo}.  This proves the following:
\begin{proposition}\label{trace.Kf.prop}
For $V \in \cinf_0(\Hn, \bbR)$, 
\[
\begin{split}
\tr \sqbrak*{f(\lap + V) - f(\lap)} &=  \int_{\Hn} \int_{-\infty}^\infty \sqbrak[\big]{K_V(\xi;z,z) - K_0(\xi,z,z)} f(\tfrac{n^2}4 + \xi^2)\> d\xi \> dg(z) \\
&\qquad + \sum_{j=1}^d f(\lambda_j).
\end{split}
\]
\end{proposition}

\section{Birman-Krein formula}\label{bk.sec}

The Birman-Krein formula relates the spectral resolution of $\lap + V$ to the scattering matrix.  This formula
provides the crucial link between the traces discussed in \S\ref{trace.sec} and the resonance set.  
The formula for the hyperbolic case is analogous to the Euclidean version \cite[Thm.~3.51]{DZbook}.

The scattering matrix associated to $V$ is defined as follows.
The Poisson operator maps a function $f \in \cinf(\bbS^n)$ to a generalized eigenfunction $E_V(s) f$, 
which admits an asymptotic expansion with leading terms
\begin{equation}\label{eigf.exp}
(2s-n) E_V(s) f \sim \rho^{n-s} f + \rho^s f',
\end{equation}
where $f' \in \cinf(\bbS^n)$ for $\re s = n/2$, $s \ne n/2$.  The structure of this expansion 
is well known and can be deduced from the resolvent identity \eqref{R0RV}.

The scattering matrix $S_V(s)$ is a family of pseudodifferential operators $S_V(s)$ on 
$\bbS^n$ that intertwines the leading coefficients of \eqref{eigf.exp},
\[
S_V(s): f \mapsto f'.
\]
For appropriate choices of $s$, we can interpret $f$ as incoming boundary data, and $f'$ as the corresponding
outgoing data.  By the meromorphic continuation of the resolvent, $S_V(s)$ extends meromorphically to $s \in \bbC$.
The identities
\begin{equation}\label{SV.inv}
S_V(s)^{-1} = S_V(n-s),
\end{equation}
and
\begin{equation}\label{ESE}
E_V(n-s)S_V(s) = -E_V(s),
\end{equation}
which follow from \eqref{eigf.exp}, hold meromorphically in $s$.

The integral kernel of the scattering matrix (with respect to the standard sphere metric)
can be derived from the resolvent by a boundary limit analogous to \eqref{EV.def},
\[
S_V(s;\omega, \omega') := (2s-n) \lim_{\rho,\rho' \to 0} (\rho\rho')^{-s} R_V(s; z,z'),
\]
for $\omega \ne \omega'$.  We can thus see from \eqref{R0RV} that
\[
S_V(s) = S_0(s) - (2s-n) E_V(s)VE_0(s). 
\]
This gives a formula for the relative scattering matrix,
\begin{equation}\label{srel.ee}
S_V(s) S_0(s)^{-1} = I + (2s-n) E_V(s)VE_0(n-s).
\end{equation}
Since $(2s-n) E_V(s)VE_0(n-s)$ is a smoothing operator, $S_V(s)S_0(s)^{-1}$ is of determinant class.
We can thus define the relative scattering determinant,
\[
\tau(s) := \det \sqbrak*{S_V(s)S_0(n-s)}.
\]
By \eqref{SV.inv}, the scattering determinant satisfies 
\begin{equation}\label{tau.refl}
\tau(s)\tau(n-s) = 1.
\end{equation}
Also, since $S_V(s)$ is unitary on the critical line,   $\abs{\tau(s)} = 1$ for $\re s = n/2$.

We can evaluate $\tau(\tfrac{n}2)$ by noting that
\begin{equation}\label{SV.n2}
S_V(\nh) = -I + 2P,
\end{equation}
where $P$ is an orthogonal projection of rank $m_V(\tfrac{n}2)$.  
(See \cite[Lemma~8.9]{B:STHS2} for the argument.)   This implies that 
\[
\tau(\tfrac{n}2) = (-1)^{m_V(\tfrac{n}2)}.
\]

The scattering phase $\sigma(\xi)$ for $\xi \in \bbR$ is defined as
\[
\sigma(\xi) :=  \frac{i}{2\pi} \log \frac{\tau(\frac{n}2 + i\xi)}{\tau(\frac{n}2)},
\]
with the branch of log chosen continuously from $\sigma(0) := 0$.  The reflection formula \eqref{tau.refl} implies that
\[
\sigma(-\xi) = -\sigma(\xi).
\]
We will be particularly interested in the derivative of the scattering phase.  By Gohberg-Krein \cite[\S IV.1]{GK:1969},
\[
\begin{split}
\frac{\tau'}{\tau}(s) &= \tr \sqbrak*{\paren*{S_V(s)S_0(n-s)}^{-1} \frac{d}{ds}\paren*{S_V(s)S_0(n-s)}} \\
&= \tr \sqbrak[\Big]{ S_V(n-s) S_V'(s) - S_0(s) S_0'(n-s)},
\end{split}
\]
where $S_V'(s) := \del_s S_V(s)$.
For the scattering phase this gives
\begin{equation}\label{logp.tau}
\sigma'(\xi) =  -\frac{1}{2\pi} \tr \sqbrak[\Big]{ S_V(\tfrac{n}2 - i\xi) S_V'(\tfrac{n}2 + i\xi) - S_0(\tfrac{n}2 + i\xi) S_0'(\tfrac{n}2 - i\xi)}.
\end{equation} 

\begin{theorem}[Birman-Krein formula]\label{BK.thm}
For $V \in L^\infty_{\rm cpt}(\Hn, \bbR)$ and $f \in \calS(\bbR)$, 
\[
\begin{split}
\tr \sqbrak[\big]{f(\lap + V) - f(\lap)} &= \int_{0}^\infty \sigma'(\xi) f(\tfrac{n^2}4 + \xi^2) \, d\xi \\
&\qquad + \sum_{j=1}^{d} f(\lambda_j) + \frac12 m_V(\tfrac{n}2) f(\tfrac{n^2}4).
\end{split}
\]
where $\lambda_1,\dots,\lambda_d$ are the eigenvalues of $\lap+V$, and $m_V(\tfrac{n}2)$ is the multiplicity of $n/2$ as
a resonance of $\lap+V$.
\end{theorem}
\begin{proof}
For convenience, let us assume that the discrete spectrum of $\lap + V$ is empty, since the contribution to the 
trace from $\lambda_1,\dots,\lambda_d$ is easily dealt with. 
Under this assumption, Proposition~\ref{trace.Kf.prop} gives
\[
\tr \sqbrak*{f(\lap + V) - f(\lap)} =  \int_{\Hn} \int_{-\infty}^\infty \sqbrak[\big]{K_V(\xi;z,z) - K_0(\xi,z,z)} f(\tfrac{n^2}4 + \xi^2)\> d\xi \> dg(z).
\]
If the integral over $z$ is restricted to the set $\set{\rho(z) \ge \vep}$, 
then switching the order of integration is justified by the uniform
polynomial bounds on $K_V$.  We can thus write
\begin{equation}\label{Ie.lim}
\tr \sqbrak*{f(\lap + V) - f(\lap)} = \lim_{\vep \to 0} \int_{-\infty}^\infty I_\vep(\xi) f(\tfrac{n^2}4 + \xi^2)\, d\xi, 
\end{equation}
where
\[
I_\vep(\xi) := \int_{\set{\rho\ge \vep}}  \sqbrak[\big]{K_V(\xi;z,z) - K_0(\xi,z,z)} d\omega'\, dg(z).
\]

To compute $I_\vep(\xi)$, we first use \eqref{RR.EE} to write $K_V$ in terms of $E_V$, 
\begin{equation}\label{Ivep1}
\begin{split}
I_\vep(\xi) &:= - \frac{(2s-n)^2}{4\pi}  \int_{\set{\rho\ge \vep}} \int_{\bbS^n} \Bigl[ E_V(s;z,\omega')E_V(n-s;z,\omega') \\
&\qquad - E_0(s;z,\omega')E_0(n-s;z,\omega') \Bigr] d\omega'\>dg(z),
\end{split}
\end{equation}
where $d\omega'$ is the standard sphere measure.  We are using the identification $s = \nh+i\xi$ freely here, 
to simplify notation where possible.
The next step is to apply a Maass-Selberg identity as described in the proof of \cite[Prop.~10.4]{B:STHS2}.  
Because $E_V(s)$ satisfies the eigenvalue equation, we can write
\[
E_V(s') E(n-s)^t = \frac{1}{s(n-s) - s'(n-s')} \sqbrak[\Big]{E_V(s') \lap E(n-s)^t - \lap E_V(s') E(n-s)^t}
\]
Applying this to \eqref{Ivep1} yields
\[
\begin{split}
I_\vep(\xi) &:=  - \frac{1}{4\pi} \lim_{s'\to s} \frac{2s-n}{s'-s} \int_{\set{\rho\ge \vep}} \int_{\bbS^n} \Bigl[ E_V(s';z,\omega') \lap E_V(n-s;z,\omega') \\
&  \qquad - \lap E_V(s'; z,\omega') E_V(n-s; z,\omega') 
- E_0(s'; z,\omega') \lap E_0(n-s; z,\omega')  \\
&\qquad - E_0(s';z,\omega') \lap E_0(n-s;z,\omega') \Bigr] d\omega'\>dg(z),
\end{split}
\]
with $\lap$ acting on the $z$ variable.  By Green's formula, applied to the region $\set{\rho = \vep}$,
\[
\begin{split}
I_\vep(\xi) &:= \frac{1}{4\pi}  \lim_{s'\to s} \frac{2s-n}{s'-s} \int_{\set{\rho = \vep}} \int_{\bbS^n} \Bigl[ E_V(s';z,\omega') \del_r E_V(n-s;z,\omega') \\
&  \qquad - \del_r  E_V(s'; z,\omega') E_V(n-s; z,\omega') 
- E_0(s'; z,\omega') \del_r  E_0(n-s; z,\omega')  \\
&\qquad - E_0(s';z,\omega') \del_r E_0(n-s;z,\omega') \Bigr]
\sinh^n r  \>d\omega'\>d\omega,
\end{split}
\]
where $z = (r,\omega)$ in geodesic polar coordinates.  The same calculation with $s=s'$ yields zero, so we can
evaluate the limit $s' \to s$ as a derivative,
\[
\begin{split}
I_\vep(\xi) &=  \frac{2s-n}{4\pi} \int_{\bbS^n} \int_{\bbS^n} 
\Bigl[  E_V'(s; z,\omega') \del _r E_V(n-s; z,\omega')  \\
&  \qquad - \del _r E_V'(s; z,\omega') E_V(n-s; z,\omega') 
- E_0'(s; z,\omega') \del _r E_0(n-s; z,\omega')  \\
&  \qquad + \del _r E_0'(s; z,\omega') E_0(n-s; z,\omega') \Bigr]
\sinh^n r  \>d\omega'\>d\omega \Big|_{r = \cosh^{-1}(1/\vep)},
\end{split}
\]
where $E_V' = \del_s E_V$.
The integrand can be simplified using the identity \eqref{SV.inv} and the distributional asymptotic 
\[
(2s-n)E_V(s;z,\omega') \sim \rho^{n-s} \delta_\omega(\omega') + \rho^s S_V(s;\omega,\omega'),
\]
which follows from \eqref{eigf.exp}.
After cancelling terms between $E_V(s)$ and $E_0(s)$, we find that
\begin{equation}\label{Iep.calc}
\begin{split}
I_\vep(\xi) &=  - \frac{1}{4\pi} \tr \Bigl[ S_V(\tfrac{n}2 - i\xi) S_V'(\tfrac{n}2 + i\xi)
- S_0(\tfrac{n}2 - i\xi) S_0'(\tfrac{n}2 + i\xi) \Bigr] \\
& \quad + \frac{\vep^{-2i\xi}}{8\pi i\xi} \tr \Bigl[ S_V(\tfrac{n}2 - i\xi)- S_0(\tfrac{n}2 - i\xi) \Bigr] \\
& \quad - \frac{\vep^{2i\xi}}{8\pi i\xi} \tr \Bigl[ S_V(\tfrac{n}2 + i\xi)- S_0(\tfrac{n}2 + i\xi) \Bigr] \\
& \quad + o(\vep).
\end{split}
\end{equation}

The first trace in \eqref{Iep.calc} reduces to $\tfrac12\sigma'(\xi)$ by \eqref{logp.tau}.  Thus, applying \eqref{Iep.calc} in \eqref{Ie.lim} gives
\begin{equation}\label{trf.alim}
\begin{split}
\tr \sqbrak[\big]{f(\lap + V) - f(\lap)} &= \frac12 \int_{-\infty}^\infty \sigma'(\xi) f(\tfrac{n^2}4 + \xi^2) d\xi \\
&\qquad + \lim_{a \to \infty} \int_{-\infty}^\infty \sqbrak*{\frac{e^{i\xi a}}{8\pi i\xi} \varphi(\xi) - \frac{e^{-i\xi a}}{8\pi i\xi} \varphi(-\xi)} d\xi,
\end{split}
\end{equation}
where we have substituted $\vep = e^{-a/2}$, and
\begin{equation}\label{varphi.def}
\varphi(\xi) := \tr \Bigl[ S_V(\tfrac{n}2 - i\xi)- S_0(\tfrac{n}2 - i\xi) \Bigr] f(\tfrac{n^2}4 + \xi^2).
\end{equation}

To evaluate the limit $a \to \infty$ in \eqref{trf.alim}, we need to control the growth of $\varphi(\xi)$.
We can argue as in \cite[Lemma~3.3]{BC:2014} that for $\chi \in \cinf_0(\bbH)$, equal to 1 on the support of $V$, 
\[
S_V(s) - S_0(s) = -(2s-n) E_0(s)^t \chi (1 + VR_0(s) \chi)^{-1} V E_0(s).
\]
The Hilbert-Schmidt norms of the cutoff factors $\chi E_0(\nh \pm i\xi)$ are $O(\abs{\xi}^{-1})$ by \cite[Lemma~3.3]{BC:2014}.
The operator norm of $(1 + VR_0(\nh \pm i\xi) \chi)^{-1}$ is $O(1)$ by the cutoff resolvent bound from Guillarmou \cite[Prop.~3.2]{Gui:2005c}.
Therefore the trace in \eqref{varphi.def} has at most polynomial growth, and $\varphi$ is integrable over $\xi \in \bbR$.

The Riemann-Lebesgue lemma gives
\[
\lim_{a \to \infty} \int_{\abs{\xi} >1} \sqbrak*{\frac{e^{i\xi a}}{8\pi i\xi} \varphi(\xi) - \frac{e^{-i\xi a}}{8\pi i\xi} \varphi(-\xi)} d\xi = 0,
\]
as well as 
\[
\lim_{a \to \infty}  \int_{-1}^1 \frac{e^{\pm i\xi a}}{8\pi i\xi} \sqbrak[\Big]{ \varphi(\pm \xi) - \varphi(0)} d\xi = 0.
\]
We can thus drop the portion of the integral
with $\abs{\xi} >1$ and replace
$\varphi(\pm\xi)$ by $\varphi(0)$ for $\abs{\xi} \le 1$ before taking the limit.  This reduces the final term in \eqref{trf.alim} to
\[
\begin{split}
\lim_{a \to \infty} \int_{-\infty}^\infty \sqbrak*{\frac{e^{i\xi a}}{8\pi i\xi} \varphi(\xi) - \frac{e^{-i\xi a}}{8\pi i\xi} \varphi(-\xi)} d\xi
&= \lim_{a \to \infty} \int_{-1}^1 \frac{\sin(\xi a)}{4\pi \xi} \varphi(0) d\xi \\
&= \varphi(0) \int_{-\infty}^\infty \frac{\sin(\xi)}{4\pi \xi} d\xi \\
&= \frac14 \varphi(0).
\end{split}
\]
To complete the argument, we note that \eqref{SV.n2} implies that
\[
\tr \Bigl[ S_V(\tfrac{n}2)- S_0(\tfrac{n}2) \Bigr] = 2m_V(\tfrac{n}2).
\]
\end{proof}

\section{Poisson formula}\label{poisson.sec}

The Poisson formula expresses the trace of the wave group as a sum over the resonance set.  
The relative wave trace,
\begin{equation}\label{ThetaV.def}
\Theta_V(t) := \tr\sqbrak*{\cos  \paren*{t\sqrt{\lap + V - n^2/4}} - \cos \paren*{t\sqrt{\lap - n^2/4}}},
\end{equation}
is defined distributionally as in \S\ref{trace.sec}.  That is, for $\psi \in \cinf_0(\bbR)$, 
\[
\paren*{\Theta_V, \psi} := \tr \sqbrak[\big]{f(\lap + V) - f(\lap)},
\]
where 
\begin{equation}\label{f.psi}
f(x) := \chi(x) \int_{-\infty}^\infty \cos\paren*{t \sqrt{x - n^2/4}} \psi(t)\>dt,
\end{equation}
with $\chi$ a smooth cutoff which equals $1$ on the spectrum of $\lap + V - n^2/4$ and vanishes on $(-\infty, c]$ for some $c<0$.
The cutoff is a technicality, included so that $f \in \calS(\bbR)$.

\begin{theorem}[Poisson formula]\label{poisson.thm}
For a potential $V \in \cinf_0(\Hn, \bbR)$, 
\begin{equation}\label{poisson.formula}
t^{n+1} \Theta_V =  t^{n+1} \sqbrak[\Bigg]{ \frac12 \sum_{\zeta\in \calR_V} e^{(\zeta-\frac{n}2)\abs{t}}
- u_0(t)}
\end{equation}
as a distribution on $\bbR$, where
\[
u_0(t) :=  \begin{dcases} \frac{\cosh t/2}{(2 \sinh t/2)^{n+1}},&\text{for }n+1\text{ even},\\
0, & \text{for }n+1 \text{ odd}. \end{dcases}
\]
\end{theorem}

A more general version of the Poisson formula for resonances for compactly supported black box perturbations of $\Hn$ was 
stated in \cite[Thm.~3.4]{Borthwick:2010}, with the proof omitted because of its similarity to the argument of
Guillop\'e-Zworski \cite{GZ:1997}.  Zworski has recently noted that the 
proof in \cite{GZ:1997} glossed over certain technical details concerning the computation of the distributional Fourier transform of the 
spectral resolution.  Furthermore, the optimal factor of $t^{n+1}$ was not obtained in these previous versions.  

The technicalities of this proof are now worked out in the book of Dyatlov-Zworski \cite[Ch.~3]{DZbook}, including the $t$ prefactor.  
The proof of \cite[Thm.~3.53]{DZbook} relies only on a global upper bound on the counting function, as in \eqref{NV.bound}, and a factorization formula
for the scattering determinant, which we state as Proposition~\ref{tau.factor.prop} below.  It therefore essentially applies 
to Theorem~\ref{poisson.thm}.   

However, there are some structural differences in
the hyperbolic case, due to the shifted spectral parameter $z = s(1-s)$ and the non-trivial background contribution of
$\bbH^{n+1}$ in even dimensions.  For the sake of completeness, we will include a hyperbolic version of the proof.  

The  starting point is to apply the Birman-Krein formula (Theorem \ref{BK.thm}) to the relative wave trace.
The relation \eqref{f.psi} implies that
\[
f(\tfrac{n^2}4 + \xi^2) =  \frac12\sqbrak*{\hat\psi(\xi) +  \hat\psi(-\xi)}.
\]
Using this, and the fact that $\sigma'(\xi)$ is even, reduces the Birman-Krein formula to
\begin{equation}\label{theta.bk}
\paren*{\Theta_V, \psi} =  \frac12 \int_{-\infty}^\infty \sigma'(\xi) \hat\psi(\xi) d\xi 
+ \sum_{j=1}^d f(\lambda_j) + \frac12 m_V(\tfrac{n}2) \hat\psi(0).
\end{equation}
To evaluate the integral in \eqref{theta.bk}, we need some additional facts about the scattering determinant.  
Given the polynomial bound on the resonance counting function \eqref{NV.bound}, we can define the Hadamard product
\[
H_V(s) := s^{m_V(0)} \prod_{\zeta \in \calR_V\backslash\set{0}} E\paren*{s/\zeta, n+1},
\]
where
\[
E(z,p) := \paren*{1 - z} e^{z + \frac{z^2}{2} + \dots + \frac{z^p}{p}}.
\]
This yields an entire function with zeros located at the resonances.

The following factorization formula provides the connection between the Birman-Krein formula and the resonance set.  
\begin{proposition}\label{tau.factor.prop}
The relative scattering determinant admits a factorization
\[
\tau(s) = (-1)^{m_V(\tfrac{n}2)} e^{q(s)} \frac{H_V(n-s)}{H_V(s)}  \frac{H_0(s)}{H_0(n-s)},
\]
where $q$ is a polynomial of degree at most $n+1$, satisfying $q(n-s) = -q(s)$.
\end{proposition}

Proposition~\ref{tau.factor.prop} is a special case of \cite[Prop.~3.1]{Borthwick:2010}, which applies to black-box perturbations 
of $\Hn$.  That statement did not include the symmetry condition on $q(s)$, which follows from \eqref{tau.refl} once
the parity of $m_V(\tfrac{n}2)$ has been factored out.
An analogous result for metric perturbations was given in \cite[Prop.~7.2]{Borthwick:2008}, without the estimate on the degree of $q(s)$.  
These previous versions contained a typo in the Hadamard product, in that the $\zeta=0$ term 
should always be treated as a separate factor $s^{m(0)}$.

In view of \eqref{rel.trace}, Theorem~\ref{BK.thm} implies that the derivative $\sigma'$ defines a tempered distribution.
We will need the following estimate of its rate of growth.
\begin{proposition}\label{temper.prop}
For $V \in \cinf_0(\Hn, \bbR)$, the derivative of the scattering phase satisfies
\[
\abs*{\sigma'(\xi)} \le C_V ( 1 + \abs{\xi})^{n-1}
\]
for $\xi \in \bbR$.
\end{proposition}

The fact that $\sigma'$ has at most polynomial growth follows from Proposition~\ref{tau.factor.prop}, 
by a general argument given in Guillop\'e-Zworski \cite[Lemma~4.7]{GZ:1997}, 
which in turn is based on a method introduced by Melrose \cite{Melrose:1988}.  The explicit growth rate
of Proposition~\ref{temper.prop} was proven in Borthwick-Crompton \cite[Prop.~3.1]{BC:2014}.

With these ingredients in place, the strategy for the proof of the Poisson formula is essentially to compute the Fourier transform 
of $\sigma'$.
\begin{proof}[Proof of Theorem~\ref{poisson.thm}]
Let us first show that the right-hand side of \eqref{poisson.formula} defines a distribution.  Indeed, if we exclude the
finite number of terms with $\re \zeta> \tfrac{n}2$, which have exponential growth,
the remaining sum gives a tempered distribution.
To see this, consider a test function $\psi \in \calS(\bbR)$.
Repeated integration by parts can be used to estimate, for $\re \zeta> \tfrac{n}2$,
\[
\abs*{\int_{-\infty}^\infty t^{n+1} e^{(\zeta-\frac{n}2) t} \psi(t) \>dt} \le \frac{C}{(1+\abs{\zeta-\tfrac{n}2})^{n+2}} \sum_{k=0}^{n+1} 
\sup_{t\in \bbR} \,\abs[\Big]{\brak{t}^{n+3} \psi^{(k)}(t)}.
\]
It then follows from the polynomial bound \eqref{NV.bound} that the sum
\[
t^{n+1} \sum_{\re \zeta < \frac{n}2} e^{(\zeta-\frac{n}2)\abs{t}}
\]
defines a tempered distribution on $\bbR$.  
The right-hand side of \eqref{poisson.formula} is thus well-defined as a distribution,
since there are only finitely many terms with $\re\zeta \ge \tfrac{n}2$.

Let $\Thetasc$ denote the tempered distribution defined by
\begin{equation}\label{vtheta.def}
(\Thetasc, \psi) := \frac12 \int_{-\infty}^\infty \sigma'(\xi) \hat\psi(\xi) d\xi.
\end{equation}
for $\psi \in \calS(\bbR)$.  This distribution accounts for the contributions to the Birman-Krein formula \eqref{theta.bk} 
from the continuous spectrum.  
The sum over the discrete spectrum can be rewritten
as a sum over the resonances with $\re s > n/2$, using the fact that
\[
\cos\paren*{t \sqrt{\lambda - n^2/4}} = \cosh( t(\zeta - \tfrac{n}2))
\]
for  $\zeta \in (\tfrac{n}2, \infty)$ and $\lambda = \zeta(n-\zeta)$. 
The Birman-Krein formula then becomes
\begin{equation}\label{Theta.vart}
\Theta_V(t) = \Thetasc(t) + \sum_{\re \zeta > \frac{n}2}  \cosh( t(\zeta - \tfrac{n}2)) + \frac12 m_V(\tfrac{n}2).
\end{equation}

Since $\Thetasc$ is tempered, it suffices to evaluate \eqref{vtheta.def} under the assumption 
that $\hat\psi \in \cinf_0(\bbR)$.  From Proposition~\ref{tau.factor.prop} we calculate
\[
\frac{\tau'}{\tau}(s) = q'(s) - \frac{H_V'}{H_V}(n-s)  - \frac{H_V'}{H_V}(s) + \frac{H_0'}{H_0}(s) + \frac{H_0'}{H_0}(n-s).
\]
The Hadamard product derivatives are given by
\[
\frac{H_V'}{H_V}(s) = \frac{m_V(0)}{s} + \sum_{\zeta \in \calR_V\backslash\set{0}} \sqbrak*{\frac{1}{s-\zeta} + \frac{1}{\zeta} + \dots + \frac{s^n}{\zeta^{n+1}}}
\]
Hence we can write 
\[
\frac{H_V'}{H_V}(n-s) + \frac{H_V'}{H_V}(s) = \sum_{\zeta \in \calR_V} \sqbrak*{\frac{n-2\zeta}{(n-s-\zeta)(s-\zeta)} + p_\zeta(s)},
\]
where $p_\zeta(s)$ is a polynomial of degree $n$ for $\zeta \ne 0$, and $p_0(s) := 0$.

Switching to a $\xi$ derivative for $\sigma$ gives
\[
\sigma'(\xi) = -\frac{1}{2\pi} \frac{\tau'}{\tau}(\nh+i\xi),
\]
which evaluates to
\begin{equation}\label{taup.sum}
\begin{split}
\sigma'(\xi) &= -\frac{1}{2\pi}q'(\tfrac{n}2 + i\xi) + \frac{1}{2\pi} \sum_{\zeta \in \calR_V} \sqbrak*{\frac{n-2\zeta}{\xi^2+ (\zeta-\frac{n}2)^2} + p_\zeta(\nh+i\xi)} \\
&\qquad -\frac{1}{2\pi} \sum_{\zeta \in \calR_0} \sqbrak*{\frac{n-2\zeta}{\xi^2+ (\zeta-\frac{n}2)^2} + p_\zeta(\nh+i\xi)},
\end{split}
\end{equation}
where $p_\zeta$ is a polynomial of degree at most $n$.  The convergence is uniform on compact intervals.  
Note that there is no pole corresponding to the possible resonance at $\zeta = \tfrac{n}2$, because a zero at this point 
would cancel out of $H_V(s)/H_V(n-s)$.

Assuming that $\hat\psi$ is compactly supported, the contributions of \eqref{taup.sum} to 
$(t^{n+1} \Thetasc, \psi)$ can be evaluated term by term.  
Under the Fourier transform, the factor $t^{n+1}$ becomes $(-i\del_\xi)^{n+1}$, which knocks out all of the polynomial terms.  
Hence, after integrating by parts,
\[
\begin{split}
(t^{n+1} \Thetasc, \psi) &= \frac{1}{4\pi} \sum_{\zeta \in \calR_V} \int_{-\infty}^\infty \frac{n-2\zeta}{\xi^2+ (\zeta-\frac{n}2)^2} (i\del_\xi)^{n+1} 
\hat\psi(\xi)\>d\xi \\
&\qquad - \frac{1}{4\pi} \sum_{\zeta \in \calR_0} \int_{-\infty}^\infty \frac{n-2\zeta}{\xi^2+ (\zeta-\frac{n}2)^2} (i\del_\xi)^{n+1} 
\hat\psi(\xi)\>d\xi.  
\end{split}
\]
By a straightforward contour integration,
\[
\int_{-\infty}^\infty e^{-i \xi t} \frac{n-2\zeta}{\xi^2+ (\zeta-\frac{n}2)^2}\>d\xi 
= \begin{dcases}
-2\pi e^{-(\zeta-\frac{n}2)\abs{t}}, & \re \zeta > \tfrac{n}2, \\
2\pi e^{(\zeta-\frac{n}2)\abs{t}}, & \re \zeta < \tfrac{n}2.
\end{dcases}
\]
Using this calculation in the formula for $(t^{n+1} \Thetasc, \psi)$ gives
\begin{equation}\label{vtheta.reson}
\begin{split}
(t^{n+1} \Thetasc, \psi) &= \frac12 \int_{-\infty}^\infty t^{n+1} \Biggl(\sum_{\zeta\in \calR_V: \re \zeta < \frac{n}2} e^{(\zeta-\frac{n}2)\abs{t}} \\
&\qquad-\sum_{\zeta\in \calR_V: \re \zeta > \frac{n}2} e^{-(\zeta-\frac{n}2)\abs{t}}  - \sum_{\zeta\in \calR_0} e^{(\zeta-\frac{n}2)\abs{t}} \Biggr)
\psi(t)\>dt.
\end{split}
\end{equation}
This calculation contains no contribution from a resonance at $\zeta = \nh$, because a zero at this point cancels out of the
formula for $\tau(s)$.

To remove the restriction of compact support for $\hat\psi$, we note that 
the right-hand side of \eqref{vtheta.reson} defines a tempered distribution by the 
remarks at the beginning of the proof.  Since $\Thetasc$ is also tempered, and 
$\cinf_0(\bbR)$ is dense in $\calS(\bbR)$, it follows that 
\eqref{vtheta.reson} holds for all $\psi \in \calS(\bbR)$.

Combining this computation of $t^{n+1}\Thetasc$ with the formula \eqref{Theta.vart} now yields the formula
\[
t^{n+1} \Theta_V =  \frac12 t^{n+1} \sqbrak[\Bigg]{ \sum_{\zeta\in \calR_V} e^{(\zeta-\frac{n}2)\abs{t}}
- \sum_{\zeta\in \calR_0} e^{(\zeta-\frac{n}2)\abs{t}}}.
\]
Note that the constant term $\tfrac12 m_V(\nh)$ from \eqref{Theta.vart} is now incorporated into the sum over $\calR_V$.
This completes the proof for $n+1$ odd, because $\calR_0$ is empty.  If $n+1$ is even, then $\calR_0$
is equal to $-\bbN_0$ as a set, with multiplicities given by the dimensions of spaces of spherical harmonics of degree $k$,
\[
m_0(-k) = (2k+n)\frac{(k+1)\dots (n+k-1)}{n!}.
\]
The resulting sum over $\calR_0$ was computed in Guillarmou-Naud \cite[Lemma 2.4]{GN:2006},
\[
\frac12 \sum_{k=0}^\infty m_0(-k) e^{-(k+\frac{n}2)\abs{t}} 
=  \frac{\cosh t/2}{(2 \sinh t/2)^{n+1}}.
\]
\end{proof}

\section{Wave trace expansion}\label{wave.exp.sec}

In this section, we compute the expansion at $t = 0$ of the relative wave trace distribution $\Theta_V$, 
as defined in \eqref{ThetaV.def},  and determine the first two wave invariants explicitly.  
Although the existence of the wave-trace expansion is considered well known, we are not aware of any direct proof for Schr\"odinger
operators in the literature.    For the odd-dimensional Euclidean case, Melrose \cite[\S4.1]{Melrose:1995} is generally cited, but this source 
does not include a proof.   Because the hyperbolic setting leads to differences from the familiar Euclidean formulas, 
we will include the argument here.

To set up the expansion formula, we recall that $\abs{t}^\beta$ is well-defined as a meromorphic family of distributions on 
$\bbR$, with poles at negative odd integers.  The residues at these poles are given by delta distributions.
Dividing by $\Gamma(\tfrac{\beta+1}2)$ cancels the poles and defines a holomorphic family, 
\begin{equation}\label{vartheta.def}
\vartheta^\beta(t) := \frac{\abs{t}^\beta}{\Gamma(\tfrac{\beta+1}2)},
\end{equation}
where 
\[
\vartheta^{-1-2j}(t) = (-1)^{j} \frac{j!}{(2j)!} \delta^{(2j)}(t),
\] 
for $j\in \bbN_0$ (see, e.g., Kanwal \cite[\S4.4, eq.~(52)]{Kanwal}).

\begin{theorem}\label{wave.exp.thm}
Let $V\in C_0^\infty(\Hn)$ with $n \geq 1$. 
For each integer $N >  [(n+1)/2]$, there exist constants $a_k(V)$ 
(the wave invariants) such that
\beq
\Theta_V(t)  =   \sum_{k = 1}^N a_k(V) \vartheta^{-n+2k-1}(t)  + F_N(t),
\eeq 
with $F_N \in C^{2N-n - 1}(\mbr)$ and $F_N(t) = O(\abs{t}^{2N-n})$ as $t \to 0$.
\end{theorem}

The proof is adapted from B\'erard \cite{Be:1977} and relies on the Hadamard-Riesz 
\cite{Hadamard:1953, Riesz:1949} construction of a parametrix for the wave kernel.  
For $V\in C_0^\infty(\Hn)$, let
\[
P_V := \lap + V - \tfrac{n^2}{4}.
\]
We denote by $e_V$ the fundamental solution of the Cauchy problem for the wave equation,
\beqq\label{eq-cauchy}
\begin{gathered}
(\p_t^2  + P_V ) e_V(t; z, w) = 0, \\
e_V(0; z,w) = \delta(z - w), \\
\p_t e_V(0; z, w) = 0,
\end{gathered}
\eeqq
for $t \in \mbr$ and $z,w\in \Hn$.  In other words,
$e_V(t;\cdot,\cdot)$ is the integral kernel of the wave operator $\cos (t\sqrt{P_V})$.

For $\alpha \in \bbC$, we define the holomorphic family of distributions
\[
\chi^\alpha_+ := \frac{x_+^\alpha}{\Gamma(\alpha+1)},
\]
using the notation of H\"ormander \cite[\S3.2]{Hormander:I}.
This family satisfies the derivative identity,
\[
\frac{d}{dx} \chi^\alpha_+ = \chi^{\alpha-1}_+.
\]
Since $\chi^0_+ = x_+$, it follows that $\chi^\alpha_+$ is a point distribution at negative integers,
\[
\chi_+^{-m} = \delta^{(m-1)}(x).
\]

For $z,w \in \Hn$ we set $r := d(z,w)$ and denote by $\chi_+^{\alpha}(t^2 - r^2)$ the pullback of $\chi^\alpha_+$ by the smooth map
$\Hn\times\Hn \times \bbR \to \bbR$ given by $(z,w,t) \mapsto t^2-d(z,w)^2$.  
Since $\chi_+^{\alpha}$ is classically differentiable for $\re \alpha > -1$, derivatives of $\chi_+^{\alpha}(t^2 - r^2)$ can be computed 
directly in this region, and then extended by analytic continuation.  Hence the formulas,
\begin{equation}\label{chia.deriv}
\begin{split}
\del_t \sqbrak*{ \chi_+^{\alpha}(t^2 - r^2)} &= 2t \chi_+^{\alpha-1}(t^2 - r^2), \\
\del_r \sqbrak*{\chi_+^{\alpha}(t^2 - r^2)} &= -2r \chi_+^{\alpha-1}(t^2 - r^2),
\end{split}
\end{equation}
are valid for all $\alpha$.

Following B\'erard \cite[\S D]{Be:1977}, we seek to construct the parametrix as a sum of the distributions $\abs{t} \chi_+^{\alpha}(t^2 - r^2)$
with increasing values of $\alpha$.  The starting point for the expansion is dictated by the initial conditions in \eqref{eq-cauchy},
so we need to understand the distributional limit of $\abs{t} \chi_+^{\alpha}(t^2 - r^2)$ as $t \to 0$.
\begin{lemma}\label{tzero.lemma}
For $\psi \in \cinf_0(\Hn)$,
\begin{equation}\label{tchi.lim}
\lim_{t\to 0} \paren[\Big]{\abs{t} \chi_+^{\alpha}\paren*{t^2 - d(z,\cdot)^2}, \psi} = 
\begin{cases} \pi^{\frac{n}2} \psi(z), & \alpha = -\nh -1, \\
0, & \alpha > -\nh - 1, \end{cases}
\end{equation}
and 
\begin{equation}\label{dtchi.lim}
\lim_{t\to 0} \paren[\Big]{\del_t \sqbrak*{\abs{t} \chi_+^{\alpha}\paren*{t^2 - d(z,\cdot)^2}}, \psi} = 0
\end{equation}
for $\alpha = -\nh - 1$ and $\alpha \ge -(n+1)/2$.
\end{lemma}
\begin{proof}
The distribution is even in the variable $t$, so it suffices to consider $t>0$.
The first formula of \eqref{chia.deriv} gives, for $k \in \bbN$,
\[
\chi_+^{\alpha}(t^2 - r^2) = \paren*{\frac{1}{2t} \del_t}^{\!k} \chi_+^{\alpha+k}(t^2 - r^2),
\]
which can be used to shift the computation to the integrable range.  For $t>0$ and $\re \alpha + k > -1$, 
we have
\[
\paren[\Big]{\chi_+^{\alpha}\paren*{t^2 - d(z,\cdot)^2}, \psi} = \frac{1}{\Gamma(\alpha+k+1)}
\paren*{\frac{1}{2t} \del_t}^{\!k} \int_0^{t} (t^2 - r^2)^{\alpha+k} \tilde{\psi}(r) r^n\>dr,
\]
where in geodesic polar coordinates $(r,\omega)$ centered at $z$,
\[
\tilde{\psi}(r) := \frac{\sinh^n r}{r^n} \int_{\bbS^n} \psi(r,\omega)\>d\omega.
\]
Rescaling $r \to tr$ in the integral gives
\begin{equation}\label{chia.psi1}
\begin{split}
&\paren[\Big]{\chi_+^{\alpha}\paren*{t^2 - d(z,\cdot)^2}, \psi} \\
&\qquad = \frac{1}{\Gamma(\alpha+k+1)}
\paren*{\frac{1}{2t} \del_t}^{\!k} \sqbrak*{ t^{2(\alpha+k)+n+1} \int_0^1 (1 - r^2)^{\alpha+k} \tilde{\psi}(tr) r^n\>dr}.
\end{split}
\end{equation}

Since $\tilde{\psi}$ is the spherical average of $\psi$ centered at $z$, and the linear term in a Taylor approximation
of $\psi$ at $z$ cancels out in this average,
\[
\begin{split}
\tilde{\psi}(r) &= \vol(\bbS^n) \psi(z) +  O(r^2)\\
&= \frac{2\pi^{\frac{n+1}2}}{\Gamma(\tfrac{n+1}2)} \psi(z) +  O(r^2).
\end{split}
\]
For the same reason, $\del_r \tilde{\psi}(r) = O(r)$. Higher radial derivatives are bounded on $\set{r>0}$.
Hence, in the leading term from \eqref{chia.psi1}, all of the $t$ derivatives are applied to the factor preceding the integral,
which gives
\[
\paren*{\frac{1}{2t} \del_t}^{\!k} \sqbrak*{t^{2(\alpha+k)+n+1}} = \frac{\Gamma(\alpha + k + \frac{n+3}2)}{\Gamma(\alpha + \frac{n+3}2)} t^{2\alpha+n+1}.
\]
The leading contribution from the $r$ integration can be calculated from Euler's beta function formula \cite[eq. (5.12.1)]{NIST}, 
\[
\int_0^1 (1 - r^2)^{\alpha+k} r^n\>dr = \frac{\Gamma(\alpha + k+ 1)\Gamma(\frac{n+1}2)}{2\Gamma(\alpha + k + \frac{n+3}2)}.
\]
Combining these results in \eqref{chia.psi1} gives, for $t>0$,
\begin{equation}\label{chia.psi2}
\paren[\Big]{\chi_+^{\alpha}\paren*{t^2 - d(z,\cdot)^2}, \psi} = \frac{\pi^{\frac{n+1}2}}{\Gamma(\alpha + \frac{n+3}2)} t^{2\alpha+n+1}
\psi(z) +  O(t^{2\alpha+n+3}).
\end{equation}
This proves \eqref{tchi.lim}, once the extra factor of $t$ has been inserted.

To establish \eqref{dtchi.lim}, we note that
\[
\del_t \sqbrak*{t \chi_+^{\alpha}\paren*{t^2 - r^2}} = \chi_+^{\alpha}\paren*{t^2 - r^2} + 2t^2 \chi_+^{\alpha-1}\paren*{t^2 - r^2},
\]
by \eqref{chia.deriv}. We thus obtain from \eqref{chia.psi2},
\[
\paren[\Big]{\del_t \sqbrak*{t\chi_+^{\alpha}\paren*{t^2 - d(z,\cdot)^2}}, \psi} = 
\pi^{\frac{n+1}2} \frac{2\alpha + n +2}{\Gamma(\alpha + \frac{n+3}2)} t^{2\alpha+n+1} \psi(z)  +  O(t^{2\alpha+n+3}),
\]
and \eqref{dtchi.lim} follows.
\end{proof}

We take the following ansatz for the parametrix:
\begin{equation}\label{evn.def}
e_{V,N}(t,z,w) := \pi^{-\frac{n}2} \sum_{k=0}^N u_{V,k}(z,w) \abs{t} \chi_+^{-\frac{n}2 + k-1}(t^2 - r^2),
\end{equation}
where $u_{V, 0}(z,z) = 1$ and 
higher coefficients $u_{V,k}$ are to be chosen so that in the expression for $(\del_t^2 + P_V)e_{V,N}(t,z,\cdot)$, the 
coefficients of $\abs{t} \chi_+^{-\frac{n}2 + k-1}(t^2 - r^2)$ cancel for $k \le N$.   
Lemma~\ref{tzero.lemma} implies that the initial conditions are satisfied,
\begin{equation}\label{initialN}
\begin{gathered}
e_{V,N}(0; z,w) = \delta(z - w), \\
\p_t e_{V,N}(0; z, w) = 0.
\end{gathered}
\end{equation}

To work out the equations for the coefficients, will compute the action of $(\del_t^2 + P_V)$ on each term.
As above, it suffices to compute for $t>0$ by evenness, and we will use the temporary abbreviations
\[
u_{V,k}(z,\cdot) \leadsto u_k, \qquad  \chi_+^{\alpha}(t^2 - r^2) \to \chi_+^{\alpha},
\]
to simplify the formulas.
The time derivatives are calculated from \eqref{chia.deriv}, 
\[
\del_t^2 \sqbrak*{t\chi_+^\alpha} = 6t \chi_+^{\alpha-1} + 4t^3 \chi_+^{\alpha-2}.
\]
Using the geodesic polar coordinate form of the Laplacian,
\begin{equation}\label{lap.polar}
\lap = -\del_r^2 - n \coth r\, \del_r + \sinh^{-2} r\, \lap_{\bbS^n},
\end{equation}
we also compute that
\[
P_V \chi_+^\alpha = (2 + 2nr \coth r)\chi_+^{\alpha-1} - 4r^2 \chi_+^{\alpha-2}.
\]
Putting these together gives
\[
\begin{split}
(\del_t^2 + P_V) \sqbrak*{u_k t\chi_+^\alpha} &= (P_Vu_k) t \chi_+^\alpha +  4r (\del_r u_k) t \chi_+^{\alpha-1}
+ \paren[\big]{8 + 2nr \coth r} u_kt\chi_+^{\alpha-1} \\
&\qquad + 4u_kt (t^2-r^2) \chi_+^{\alpha-2}.
\end{split}
\]
The final term simplifies,
\[
(t^2-r^2) \chi_+^{\alpha-2} = (\alpha-1) \chi_+^{\alpha-1},
\]
which reduces the formula to 
\[
(\del_t^2 + P_V) \sqbrak*{u_k t\chi_+^\alpha} = (P_Vu_k) t \chi_+^\alpha +  
\sqbrak[\Big]{4r \del_r u_k + \paren*{4(\alpha+1) + 2nr \coth r} u_k} t\chi_+^{\alpha-1}.
\]
After setting $\alpha = -\nh+k-1$ as in \eqref{evn.def}, we obtain, for $t>0$,
\begin{equation}\label{evn.terms}
\begin{split}
(\del_t^2 + P_V) \sqbrak*{u_k t\chi_+^{-\nh+k-1}} &= 
\sqbrak[\Big]{4r \del_r u_k + \paren*{2n(r \coth r-1) + 4k} u_k} t\chi_+^{-\nh+k} \\
&\qquad + (P_Vu_k) t \chi_+^{-\nh+k-1}.  
\end{split}
\end{equation}

The calculation \eqref{evn.terms} shows the cancelling of terms in $(\del_t^2 + P_V)e_{V,N}(t,z,\cdot)$ is
ensured by the transport equations:
\begin{equation}\label{transport.eq}
\begin{split}
\sqbrak[\big]{4r \del_r + 2n (r \coth r - 1)} u_{V,0}(z,\cdot) &= 0. \\
\sqbrak[\big]{4r \del_r + 2n(r \coth r - 1) + 4k} u_{V,k}(z,\cdot) &=  - P_V u_{V,k-1}(z,\cdot). \\
\end{split}
\end{equation}
To solve \eqref{transport.eq} we define
\begin{equation}\label{phi.def}
\phi(r) := \paren*{\frac{\sinh r}{r}}^{\!-\frac{n}2},
\end{equation}
and then set
\begin{equation}\label{transport.soln}
\begin{split}
u_{V,0}(z,w) &= \phi(r), \\
u_{V,k+1}(z,w) &=  -\frac14 \phi(r) \int_0^1 \frac{s^k}{\phi(sr)}  P_V u_{V,k}(z, \gamma(s)) ds,
\end{split}
\end{equation}
where $\gamma$ is the geodesic from $z$ to $w$, parametrized by $s \in [0,1]$, and $P_V$ acts on the second variable of $u_{V,k}$.
The coefficients $u_{V,k}$ are smooth for all $k$.  

\begin{proposition}\label{wave.parametrix.prop}
With $e_{V,N}$ defined as above, set
\[
q_{V,N}(t,z,w) := e_V(t,z,w) - e_{V,N}(t,z,w).
\]
For $m \in \bbN$, we have $q_{V,N} \in C^m$ for $N$ sufficiently large and
\[
\abs{q_{V,N}(t,z,w)} = O(\abs{t}^{-n+N-1})
\]
as $t \to 0$, uniformly in $z,w$.
\end{proposition}
\begin{proof}
From \eqref{evn.def}, \eqref{evn.terms}, and the transport equations \eqref{transport.eq}, we observe that 
\[
(\del_t^2 + P_V) e_{V,N}(t,z,\cdot) = f_N(t,z,\cdot),
\]
where
\begin{equation}\label{errorN}
f_N(t,z,\cdot) = \pi^{-\frac{n}2} P_V u_{V,N}(z,\cdot)  \abs{t}\,  \chi_+^{-\frac{n}2 + N-1}(t^2 - r^2),
\end{equation}
with $P_V$ acting on the second variable.
Since $e_{V,N}$ satisfies the same initial conditions \eqref{initialN} as $e_V$, this gives
\[
\begin{gathered}
(\p_t^2  + P_V ) q_{V,N}(t; z, w) = f_N(t,z,w), \\
q_{V,N}(0; z,w) = 0, \\
\p_t q_{V,N}(0; z, w) = 0.
\end{gathered}
\]
The coefficients $u_{V,k}$ are smooth, by \eqref{transport.soln}, and $\abs{t} \chi_+^\alpha(t^2 - r^2)$ is $C^l$ for $\alpha > l+1$. 
Hence, by \eqref{errorN}, $f_N \in C^l$ for $l < N - \nh -2$ and has support in $\set{r\le t}$.  It follows that $q_V \in C^m$
for $N$ sufficiently large.  
 
For any $b>0$, the Sobolev norms of $f_{N}(t, z,\cdot)$ can be estimated by $O(\abs{t}^b)$ for $N$ sufficiently large.  These estimates are 
uniform in $z$, since $\phi$ depends only on $r$ and $V$ has compact support.  
Standard regularity estimates for hyperbolic PDE (see for example \cite[Ch. 47]{Treves})
then show that for $N_1$ sufficiently large,
\[
\abs{q_{V,N_1}(t,z,w)} = O(\abs{t}^{2N - n - 1}),
\]
uniformly in $z,w$.   The estimate of $q_{V,N}$ as $t \to 0$ is then derived from
\[
q_{V,N}(t,z,w) =  \pi^{-\frac{n}2} \sum_{k=N+1}^{N_1}  u_{V,k}(z,w) \abs{t} \chi_+^{-\frac{n}2 + k-1}(t^2 - r^2) + q_{V,N_1}(t,z,w).
\]
\end{proof}

With this estimate on the parametrix, we are now prepared to establish the wave trace expansion. 

\begin{proof}[Proof of Theorem \ref{wave.exp.thm}]
For $\psi \in \cinf_0(\bbR)$, we write the integral kernel of the trace-class operator
\[
\int_{-\infty}^\infty \sqbrak[\Big]{\cos(t\sqrt{P_V}) - \cos(t\sqrt{P_0})} \psi(t) dt
\]
as
\[
\int_{-\infty}^\infty \sqbrak[\big]{e_V(t,z,w) - e_0(t,z,w)} \psi(t) dt,
\]
which is smooth and compactly supported.  Taking the trace gives
\begin{equation}\label{Theta.ev}
(\Theta_V, \psi) = \int_{\Hn}\paren*{ \int_{-\infty}^\infty \sqbrak[\big]{e_V(t,z,w) - e_0(t,z,w)} \psi(t) dt  }\bigg|_{z=w} dg(z)
\end{equation}
The wave kernel parametrices can be substituted into \eqref{Theta.ev} and 
the contributions from the terms $k=0, \dots, N$ evaluated separately.  

For $\re \alpha$ sufficiently large we can verify directly that
\[
\abs{t}\, \chi_+^\alpha(t^2 - r^2)\Big|_{r=0} = \vartheta^{2\alpha+1}(t),
\]
and this formula extends to all $\alpha \in \bbC$ by analytic continuation.   It follows from \eqref{Theta.ev} and Proposition~\ref{wave.parametrix.prop} that
\[
\Theta_V(t) = \pi^{-\frac{n}2} \sum_{k=1}^N \paren*{\int_{\Hn} \sqbrak[\big]{u_{V,k}(z,z) - u_{0,k}(z,z)} dg(z)} 
\vartheta^{-n + 2k-1}(t) + F_N(t),
\]
where 
\[
F_N(t) := \int_{\Hn} \sqbrak[\big]{q_{V,N}(t,z,z) - q_{0,N}(t,z,z)} dg(z).
\]
\end{proof}

The proof of Theorem~\ref{wave.exp.thm} yields a formula for the wave invariants,
\begin{equation}\label{akv.explicit}
a_k(V) = \pi^{-\frac{n}2} \int_{\Hn} \sqbrak[\big]{u_{V,k}(z,z) - u_{0,k}(z,z)} dg(z). 
\end{equation}
This formula can be simplified somewhat using the transport equations.  By \eqref{transport.eq},
we have
\begin{equation}\label{transport.iter}
(L+4)\cdots (L + 4k) u_{V,k}(z,\cdot) = (-1)^k P_V u_{V,0}(z,\cdot),
\end{equation}
where, in geodesic polar coordinates centered at $z$, $L$ is the differential operator,
\[
L := 4r \del_r + 2n(r \coth r - 1).
\]
Note that for any smooth function $f$, $Lf$ vanishes at $r=0$.  Therefore, evaluating \eqref{transport.iter} at the point $z$, yields
\begin{equation}\label{uk.zero}
u_{V,k}(z,z) = \frac{(-1)^k}{4^kk!} P_V^k u_{V,0}(z,z).
\end{equation}
where $u_{V,0}(z,w) = \phi(d(z,w))$ and $P_V^k$ acts on the second variable.
In principle, \eqref{uk.zero} can be used to derive explicit formulas for all of the wave invariants.  The first two are relatively simple.
\begin{proposition}\label{d12.prop}
For $V \in \cinf_0(\Hn, \bbR)$, 
\[
a_1(V) = -\frac14 \pi^{-\frac{n}2} \int_{\Hn} V(z)\>dg(z)
\]
and
\[
a_2(V) = \frac1{32} \pi^{-\frac{n}2} \int_{\Hn} \sqbrak*{ \frac{2n-n^2}{6} V(z) + V(z)^2}\>dg(z).
\]
\end{proposition}
\begin{proof}
Since $u_{V,0}(z,z) = 1$ and $P_V - P_0 = V$, we see immediately from \eqref{uk.zero} that
\[
u_{V,1}(z,z) - u_{0,1}(z,z) = -\frac14 V(z).
\]
This gives the formula for $a_1(V)$.

For the second invariant, we use \eqref{uk.zero} to write
\[
\begin{split}
u_{V,2}(z,z) - u_{0,2}(z,z) &= \frac{1}{32} \paren*{P_V^2\phi - P_0^2\phi}\Big|_{r=0} \\
&=  \frac{1}{32} \sqbrak*{2V(z)P_0\phi(0) + \lap V(z) + V(z)^2},
\end{split}
\]
where $r$ is the radius for geodesic polar coordinates centered at $z$, and we have used the facts that
$\phi(0)=1$ and $\del_r\phi(0)=0$.
From \eqref{lap.polar} and \eqref{phi.def} we compute that
\[
\begin{split}
P_0 \phi(0) &= \frac{n(n+1)}6 - \frac{n^2}4 \\
&= \frac{2n-n^2}{12}.
\end{split}
\]
This gives
\[
u_{V,2}(z,z) - u_{0,2}(z,z) = \frac{1}{32} \sqbrak*{ \frac{2n-n^2}{6} V(z) + V(z)^2 + \lap V(z)}.
\]
When substituted into the formula for $a_2(V)$, the term $\lap V(z)$ integrates to zero, 
because $V$ has compact support.
\end{proof}

\section{Heat trace}\label{heat.sec}

The relative heat trace associated to a potential $V \in \cinf_0(\Hn, \bbR)$ is defined by 
applying the distribution \eqref{rel.trace} to the function $f(x) = \chi(x)e^{-t(x-n^2/4)}$ for $t > 0$,
where $\chi$ is a smooth cutoff which equals $1$ on the spectrum of $P_V := \lap - \tfrac{n^2}4 + V$ 
and vanishes on $(-\infty, c]$ for some $c<0$.
The Birman-Krein formula (Theorem~\ref{BK.thm}) gives
\begin{equation}\label{heat.bk}
\begin{split}
\tr \sqbrak*{e^{-t P_V} - e^{- tP_0}} &= \int_{0}^\infty \sigma'(\xi) e^{-\xi^2 t} d\xi \\
& \qquad + \sum_{j=1}^d e^{t(\frac{n^2}4 - \lambda_j)}  + \tfrac12 m_V(\tfrac{n}2).
\end{split}
\end{equation}
We have no analog of the Poisson formula of Theorem~\ref{poisson.thm} for the heat trace.  This is because the
values of $(\zeta-\tfrac{n}2)^2$ are spread over the full complex plane, so there is no apparent 
regularization of the heat trace as a sum over the resonance set.  

The asymptotic expansion of the heat trace at $t=0$ can be derived by a variety of methods.  
The simplest route for us via the wave trace expansion.
\begin{theorem}\label{heat.trace.thm}
As $t \to 0$, the relative heat trace admits an asymptotic expansion
\begin{equation}\label{heat.trace.PV}
\tr \sqbrak*{e^{-t P_V} - e^{- tP_0}} \sim \pi^{-\frac12}  \sum_{k=1}^\infty a_k(V) (4t)^{-\frac{n+1}2+k},
\end{equation}
where $a_k(V)$ are the wave invariants from Theorem~\ref{wave.exp.thm}.
\end{theorem}

\begin{proof}
Since $\Thetasc(s)$ is tempered, the definition \eqref{vtheta.def} gives
\[
\frac{1}{\sqrt{4\pi t}} \int_{-\infty}^\infty \Thetasc(s) e^{-s^2/4t} ds = 
 \int_{0}^\infty \sigma'(\xi) e^{-\xi^2 t} d\xi,
\]
for $t>0$.  It then follows from \eqref{Theta.vart} and \eqref{heat.bk} that
\begin{equation}\label{htr.theta}
\tr \sqbrak*{e^{-tP_V} - e^{-tP_0}} =  \frac{1}{\sqrt{4\pi t}} \int_{-\infty}^\infty \Theta_V(s) e^{-s^2/4t} ds,
\end{equation}
even when $\Theta_V(s)$ is not tempered.

For $\re \beta$ sufficiently large, we compute
\[
\frac{1}{\sqrt{4\pi t}} \int_{-\infty}^\infty \vartheta^\beta(s) e^{-s^2/4t} ds = \pi^{-\frac12} (4t)^{\frac{\beta}2},
\] 
and this formula extends to all $\beta\in \bbC$ by analytic continuation.
Theorem~\ref{wave.exp.thm} thus yields the expansion
\[
\tr \sqbrak*{e^{-tP_V} - e^{-tP_0}} = \pi^{-\frac12} \sum_{k=1}^N a_k(V)  (4t)^{-\frac{n+1}2+k} + 
\frac{1}{\sqrt{4\pi t}} \int_{-\infty}^\infty F_N(s)e^{-s^2/4t}\>ds,
\]
for $N > [(n+2)/2]$, where $F_N(s) = O(\abs{s}^{2N-n})$ as $s \to 0$.  From \eqref{Theta.vart} we can also see that 
$F_N(s) = O(e^{n\abs{s}/2})$ as $\abs{s} \to \infty$.  It follows that 
\[
\frac{1}{\sqrt{4\pi t}} \int_{-\infty}^\infty F_N(s)e^{-s^2/4t}\>ds = O(t^{-\frac{n}2 + N}).
\]
\end{proof}

Note that coefficients in \eqref{heat.trace.PV} are not quite the usual heat invariants, because of the shift $-n^2/4$ in the definition of $P_V$.  
This shift gives an extra factor of $e^{-n^2t/4}$ in \eqref{heat.trace.PV}, so the traditional heat invariants could be computed 
as finite linear combinations of the $a_k(V)$.

The behavior of the heat kernel as $t\to \infty$ is also of interest to us.  According to \eqref{heat.bk}, this behavior is dominated by 
exponential terms corresponding to eigenvalues and a constant term from the possible resonance $s = n/2$.  If these contributions are absent, then the heat kernel 
decays at a rate independent of the dimension.

\begin{proposition}\label{hk.decay.prop}
Suppose that for the potential $V \in \cinf(\Hn, \bbR)$, $P_V$ has no eigenvalues and no resonance at $n/2$.  Then
the following bound holds uniformly for $t\in (0,\infty)$ and $z,w$ in $\bbH$,
\begin{equation}\label{heatV.bound}
e^{-tP_V}(t;z,w) \asymp t^{-\frac{n+1}2} e^{-r^2/4t-nr/2}(1+r+t)^{\frac{n}2-1} (1+r),
\end{equation}
where $r = d(z,w)$ and $\asymp$ means that the ratio of the two sides is bounded above and below by positive constants.
\end{proposition}

In the case $V=0$, \eqref{heatV.bound} was proven in Davies-Mandouvalos \cite[Thm.~3.1]{DaviesMand:1988}. 
There is no factor $e^{-n^2t/4}$ in \eqref{heatV.bound} because of the shift $-n^2/4$ in the definition of $P_V$.
These estimates were generalized in Chen-Hassell \cite[Thm.~5]{ChenHassell:2020} to asymptotically 
hyperbolic Cartan-Hadamard manifolds with no eigenvalues and no resonance at $s = n/2$, by methods that 
allow for the inclusion of a $\cinf_0$ potential.  The power $t^{-3/2}$, independent of dimension, corresponds
to the vanishing of the spectral resolution $K_V(\xi;\cdot,\cdot)$ to order $\xi^2$ at $\xi =0$, under the assumption
of no resonance at $n/2$.

Proposition~\ref{hk.decay.prop} implies a bound on the heat trace, by 
the argument from S\'a Barreto-Zworski \cite[Prop.~3.1]{SaZw}.  
\begin{corollary}\label{ht.decay.cor}
For a potential $V$ satisfying the hypotheses of Proposition~\ref{hk.decay.prop},
\[
\tr \sqbrak*{e^{-t P_V} - e^{- tP_0}} = O(t^{-1/2})
\]
as $t \to \infty$.
\end{corollary}
\begin{proof}
Duhamel's principle gives the trace estimate
\begin{equation}\label{duhamel}
\begin{split}
&\abs*{\tr \sqbrak*{e^{-t P_V} - e^{- tP_0}}} \\
&\quad\le \int_0^t \int_{\Hn} \int_{\Hn} e^{-(t-s) P_0}(w,z) e^{-sP_V}(z,w) \abs{V(z)}\>dg(z) dg(w) ds.
\end{split}
\end{equation}
The uniform bound \eqref{heatV.bound} implies in particular that 
\[
e^{-sP_V}(z,w) \le C_Ve^{-sP_0}(z,w),
\]
and so we can use the semigroup property to estimate
\[
\begin{split}
\int_{\Hn}e^{-(t-s) P_0}(w,z) e^{-sP_V}(z,w) \> dg(w) &\le C_Ve^{-tP_0}(z,z) \\
&\le C_Vt^{-3/2}
\end{split}
\] 
as $t \to \infty$, uniformly in $z$. Applying this to \eqref{duhamel} gives
\[
\abs*{\tr \sqbrak*{e^{-t P_V} - e^{- tP_0}}} \le C_V t^{-1/2} \norm{V}_{L^1} .
\] 
\end{proof}

\section{Scattering phase asymptotics}\label{scphase.sec}

The Birman-Krein formula allows us to connect the wave-trace invariants to corresponding 
asymptotic expansions for the scattering phase and its derivative.
For Schr\"odinger operators in the odd-dimensional Euclidean setting, 
the asymptotic expansion of the scattering phase was established 
by Colin de Verdi\`ere \cite{Colin:1981}, Guillop\'e \cite{Guillope:1981}, and Popov \cite{Popov:1982},
via formulas relating the scattering determinant to regularized determinants of the cutoff resolvent.  
An argument based on expansion of the scattering matrix is given in Yafaev \cite[Thm.~9.2.12]{Yafaev:2010},
and a semiclassical version in Dyatlov-Zworski \cite[Thm.~3.62]{DZbook}.

For hyperbolic space we have the following version of these results:  

\begin{theorem}\label{sigmap.thm}
For $V \in \cinf_0(\Hn, \bbR)$ the function $\sigma'(\xi)$ admits a full asymptotic expansion as $\xi \to +\infty$.
If the dimension $n+1$ is odd, then
\[
\sigma'(\xi) \sim \sum_{k=1}^{\infty} c_k(V) \xi^{n-2k}.
\]
For $n+1$ even, the expansion is truncated,
\[
\sigma'(\xi) = \sum_{k=1}^{[n/2]}c_k(V) \xi^{n-2k} + O(\xi^{-\infty}).
\]
The coefficients are related to the wave invariants by
\[
c_k(V) = \frac{2^{-n+2k}}{\pi^{\frac12} \Gamma(\frac{n+1}{2}-k)} a_k(V).
\]
\end{theorem}

Before proving the theorem, we start by establishing the existence of the scattering phase expansion.  
The coefficients are relatively easy to calculate once this is known.

\begin{proposition}\label{sigma.exp.prop}
For $V \in \cinf_0(\Hn, \bbR)$ the function $\sigma'(\xi)$ admits an asymptotic expansion of the form
\begin{equation}\label{sigmap.exp}
\sigma'(\xi) \sim \sum_{j=0}^\infty b_j \xi^{n-j-1},
\end{equation}
as $\xi \to \infty$.
\end{proposition}

\begin{proof}
Of the approaches mentioned above, the ray expansion method from Yafaev \cite[\S8.4]{Yafaev:2010} is the most easily
adapted to the hyperbolic setting. In our context, the idea is to expand $E_V(s;z,\omega')$ in powers of $s$,
and then apply this expansion to the scattering phase.

To develop the approximation formula, we first consider $z \in \bbH^{n+1}$ with $\omega' = \infty$.
In standard hyperbolic coordinates $z = (x,y) \in \bbR^n \times \bbR_+$,
\begin{equation}\label{lap.Hn}
\lap = -y^2\del_y^2 +(n-1)y\del_y + y^2\lap_x,
\end{equation}
and the unperturbed generalized eigenfunction has the form (see, e.g., \cite[\S4]{Borthwick:2010})
\begin{equation}\label{E0.infty}
E_0(s;z,\infty) = 2^{-2s-1} \pi^{-\frac12} \frac{\Gamma(s)}{\Gamma(s-\frac{n}2+1)} y^s.
\end{equation}
In geodesic coordinates, $y = e^{-r}$, so this is the analog of a Euclidean plane wave with frequency $\xi = \im s$.  

Following the construction in \cite[\S8.1]{Yafaev:2010}, we define an approximate plane wave using the ansatz
\begin{equation}\label{psiN.def}
\psi_N(s;z) = \sum_{j=0}^N s^{-j} y^s w_j(z),
\end{equation}
with $w_0(z)=1$.  From \eqref{lap.Hn}, we have
\[
\sqbrak[\big]{\lap + V - s(n-s)} (y^s w_j) = y^s (\lap + V) w_j - 2sy^{s-1}\del_y w_j.
\]
We can thus cancel coefficients up to order $s^N$ by imposing the transport equation
\[
2y \del_y w_{j+1} = (\lap + V) w_j.
\]
The solutions are given recursively by
\begin{equation}\label{bj.formula}
w_{j+1}(z) := \frac12 \int_{-\infty}^0 (\lap + V)w_j(x,e^ty) dt 
\end{equation}
for $j \ge 1$. With these coefficients, the function \eqref{psiN.def} satisfies
\begin{equation}\label{bN.remainder}
\sqbrak[\big]{\lap + V - s(n-s)}\psi_N(s;z) = s^{-N}y^s w_N(z).
\end{equation}

In \eqref{bj.formula}, the point $(x,e^t)$ can be interpreted geometrically as the translation of $z=(x,y)$ by distance $t$
along the vertical geodesic through $z$.  Returning to the geodesic polar coordinates $z = (r,\omega) \in \bbR_+ \times \bbS^n$
used to define $E_V(s)$, we let $\phi_{z,\omega'}(t)$ denote the unique geodesic through $z$ with limit point 
$\omega' \in \bbS^n$ as $t \to 0$.  Let $w_0(z,\omega') = 1$ and define $w_j(z,\omega')$ for $j\ge 1$ by 
\[
w_{j+1}(z,\omega') := \frac12 \int_{-\infty}^0 (\lap + V)w_j(\phi_{z,\omega'}(t)) dt.
\]
For the approximate Poisson kernel,
\begin{equation}\label{EVN}
E_{V,N}(s;z,\omega') := \sum_{j=0}^N s^{-j} w_j(z,\omega') E_0(s,z,\omega'),
\end{equation}
the calculation of \eqref{bN.remainder} shows that
\[
\sqbrak[\big]{\lap + V - s(n-s)}E_{V,N}(s;z,\omega') := s^{-N} E_0(s,z,\omega') (\lap + V) w_N(z,\omega').
\]

The coefficients of \eqref{EVN} have support properties analogous to the approximate plane waves in the Euclidean case.  That is,  
for $j \ge 1$, $w_j(z,\omega')$ vanishes unless $z$ lies on a geodesic connecting a point in $\supp V$ to the limit point $\omega'$.  
One can thus repeat the argument from \cite[Thm.~8.4.3]{Yafaev:2010}, using the
cutoff resolvent bound from Guillarmou \cite[Prop.~3.2]{Gui:2005c} in place of its Euclidean counterpart.
The result is that
\begin{equation}\label{EV.approx}
E_V(s;z,\omega') = E_{V,N}(s;z,\omega') + q_n(s;z,\omega'),
\end{equation}
where, for $\re s = \nh$,
\[
\norm{q_n(s;\cdot,\omega')}_{L^2(B)} = O(s^{\frac{n}2-N}),
\]
with $B$ a ball in $\Hn$ containing $\supp V$.
The shift in the power in the error estimate comes from the Gamma factors in the normalization of \eqref{E0.infty}.
The same error estimate applies when \eqref{EV.approx} is differentiated with respect to $s$.

The approximation \eqref{EV.approx} can be applied to the scattering phase through the formula
\eqref{srel.ee}, which gives
\[
\tau(s) = \det (1 + T(s)),
\]
where
\[
T(s) := (2s-n) E_V(s)VE_0(n-s).
\]
By the definition of the scattering phase, and the fact that $1 + T(\nh+i\xi)$ is unitary for $\xi \in \bbR$,
\begin{equation}\label{sigp.trace}
\sigma'(\xi) = -\frac{1}{2\pi} \tr \sqbrak*{(1 + T(\nh+i\xi)^*)T'(\nh+i\xi)}.
\end{equation}
The kernels of $T(s)$ and $T'(s)$ are smooth, and \eqref{EV.approx} gives uniform asymptotic expansions of their kernels 
for $\re s = \nh$, with leading term of order at most $\xi^{n-1}$.  We can thus deduce the expansion of $\sigma'(\xi)$
from \eqref{sigp.trace}.
\end{proof}

Although the leading term in \eqref{sigmap.exp} matches the growth estimate of Proposition~\ref{temper.prop},
this coefficient vanishes and the leading order is actually $\xi^{n-2}$.
Computing coefficients through the construction of Proposition~\ref{sigma.exp.prop} is rather cumbersome, however.   
There is a much easier method, by comparison to the heat trace expansion via the Birman-Krein formula.

\begin{proof}[Proof of Theorem \ref{sigmap.thm}]
By a straightforward calculus argument (see \cite[Lemma~3.65]{DZbook}), 
the expansion \eqref{sigmap.exp} yields the corresponding expansion,
\[
\begin{split}
\int_{0}^\infty \sigma'(\xi) e^{-\xi^2 t} d\xi &\sim \frac12 \sum_{j=0}^{n-1} \Gamma(\tfrac{n-j}{2}) b_j \, t^{-\frac{n-j}2} 
- \frac{1}{2} \sum_{l=0}^{\infty} \frac{(-1)^l}{l!} b_{n+2l}\, t^l \log t  \\
&\qquad + \frac12 \sum_{l=0}^{\infty} \Gamma(-l-\tfrac{1}{2}) b_{n+2l+1} t^{l + \frac12} + g(t),
\end{split}
\]
as $t \to 0^+$, where $g \in C^\infty[0,\infty)$.  The function $g$ is not determined by the coefficients $b_j$.  
On the other hand, by \eqref{heat.bk} and Theorem~\ref{heat.trace.thm} we have
\begin{equation}\label{heat.sigmap}
\int_{0}^\infty \sigma'(\xi) e^{-\xi^2 t} d\xi \sim \pi^{-\frac12} \sum_{k=1}^{\infty}  a_k(V) (4t)^{-\frac{n+1}2+k} 
+ h(t).
\end{equation}
where $h \in C^\infty[0,\infty)$ is given by
\[
h(t) := \sum_{j=1}^d e^{t(\frac{n^2}4 - \lambda_j)}  + \tfrac12 m_V(\tfrac{n}2)
\]

If $n+1$ is odd, then comparing these expansions shows that $b_j = 0$ if $j$ is even, and
\begin{equation}\label{bj.odd}
b_{2k-1} = \frac{2^{-n+2k}}{\pi^{\frac12} \Gamma(\frac{n+1}{2}-k)} a_k(V)
\end{equation}
for $k \in \bbN$.
For $n+1$ even the heat trace expansion contains only integral powers of $t$.
This implies that $b_j = 0$ for all $j \ge n$, and also for even values of $j<n$.    
For odd values of $j < n$, the coefficients are given by \eqref{bj.odd}.
\end{proof}

Integrating the asymptotic expansion from Theorem~\ref{sigmap.thm} yields the following:
\begin{corollary}\label{sigma.cor}
The scattering phase admits a full asymptotic expansion as $\xi \to 0$.
If the dimension $n+1$ is odd, then
\[
\sigma(\xi) \sim \sum_{k=1}^{[n/2]} \frac{c_k(V)}{n-2k+1} \xi^{n-2k+1} + d + \tfrac12 m_V(\tfrac{n}2) 
+ \sum_{k>[n/2]} \frac{c_k(V)}{n-2k+1} \xi^{n-2k+1},
\]
where $d$ is the number of eigenvalues.  For $n+1$ even,
\[
\sigma(\xi) = \sum_{k=1}^{[n/2]} \frac{c_k(V)}{n-2k+1} \xi^{n-2k+1} + d + \tfrac12 m_V(\tfrac{n}2) 
+ O(\xi^{-\infty}).
\]
\end{corollary}
\begin{proof}
By Theorem~\ref{sigmap.thm}, the function
\[
t \mapsto \int_0^\infty \sqbrak[\Bigg]{\sigma'(x) - \sum_{k=1}^{[n/2]}c_k(V) x^{n-2k}} e^{-x^2t} dx
\]
is continuous for $t \in [0,\infty)$.
By \eqref{heat.sigmap} taking the limit as $t \to 0^+$ yields
\[
\int_0^\infty \sqbrak[\Bigg]{\sigma'(x) - \sum_{k=1}^{[n/2]}c_k(V) x^{n-2k}} dx = d + \tfrac12 m_V(\tfrac{n}2).
\]
Splitting the integral at $x=\xi$ then gives, since $\sigma(0) =0$,
\[
\begin{split}
\sigma(\xi) &= \sum_{k=1}^{[n/2]} \frac{c_k(V)}{n-2k+1} \xi^{n-2k+1} 
+ d + \tfrac12 m_V(\tfrac{n}2)\\
&\qquad - \int_\xi^\infty \sqbrak[\Bigg]{\sigma'(x) - \sum_{k=1}^{[n/2]}c_k(V) x^{n-2k}} dx.
\end{split}
\]
By Theorem~\ref{sigmap.thm}, for $n+1$ odd the final integral on the right can be integrated
to produce an asymptotic expansion in $\xi$.  For $n+1$ even, this integral gives an
error term  $O(\xi^{-\infty})$.
\end{proof}

\section{Existence of resonances}\label{exist.sec}

The asymptotic expansions of the wave trace and scattering have significantly different behavior in
odd and even dimensions, so we will consider the two cases separately.

\subsection{Even dimensions}
For $n+1$ even, all of the singularities 
in the wave trace expansion of Theorem~\ref{wave.exp.thm} are detectable for $t \ne 0$.  
It thus follows immediately from Theorem~\ref{poisson.thm} that for $V \in \cinf_0(\Hn, \bbR)$ 
the resonance set $\calR_V$ determines all of the wave invariants $a_k(V)$.  
In particular, since the vanishing of the first two wave invariants 
implies $V=0$ by the formulas of Proposition~\ref{d12.prop}, we obtain the following:
\begin{theorem}\label{even.R0.thm}
For $V \in \cinf_0(\Hn, \bbR)$ with $n+1$ even, if $\calR_V = \calR_0$ then $V=0$.
\end{theorem}

We can also deduce a lower bound on the resonance counting function from the wave trace in even dimensions.
Note that $\Theta_V(t) = O(t^{-n+1})$ by Theorem~\ref{wave.exp.thm}, whereas 
the $\calR_0$ contribution in \eqref{poisson.formula} satisfies
\[
u_0(t)  \sim \frac{1}{t^{n+1}}
\]
as $t \to 0$.
It thus follows from \eqref{poisson.formula} that
\begin{equation}\label{even.trace.asym}
\sum_{\zeta \in \calR_V} e^{(\zeta-\frac{n}2)t}  \sim \frac{2}{t^{n+1}}.
\end{equation}
The lower-bound argument from Guillop\'e-Zworski \cite[Thm.~1.3]{GZ:1997} (see also \cite[\S12.2]{B:STHS2})
can be applied to \eqref{even.trace.asym}, yielding the following:
\begin{theorem}
For $n+1$ even, the counting function for $\calR_V$ satisfies
\[
N_V(r) \ge c r^{n+1},
\] 
for some constant $c>0$ that depends only on $n$ and the radius of $\supp V$.
\end{theorem}

\begin{proof}
Choose $\phi \in \cinf_0(\bbR_+)$ with $\phi \ge 0$ and $\phi(1) >0$, and set 
\[
\phi_\lambda(t) := \lambda \phi(\lambda t).
\]
By \eqref{even.trace.asym} we have
\[
\int_0^\infty \paren[\bigg]{\sum_{\zeta \in \calR_V} e^{(\zeta-\frac{n}2)t}} \phi_\lambda(t)\>dt 
\ge c_n \lambda^{n+1},
\]
where $c_n$ does not depend on $V$. Using the Fourier transform to evaluate the right hand-side gives
\begin{equation}\label{hatphi.sum}
\sum_{\zeta \in \calR_V} \hat{\phi}\paren*{i(\zeta-\nh)/\lambda} \ge c_n \lambda^{n+1}.
\end{equation}
Since $\hat\phi(\xi)$ is rapidly decreasing, we can estimate $\hat\phi(\xi) = O(\abs{\xi}^{-n-2})$ in particular.
In terms of the counting function, \eqref{hatphi.sum} then implies that
\[
\begin{split}
c_n\lambda^{n+1} & \le  \int_0^\infty (1+r/\lambda)^{-n-2}\>dN_V(r)\\
& =  (n+2) \int_0^\infty  (1+r)^{-n-3}N_V(\lambda r)\>dr.
\end{split}
\]
Splitting the integral at $r=a$ and adjusting the constant gives
\begin{equation}\label{cnl.split}
c_n\lambda^{n+1} \le  N_V(\lambda a) + \int_a^\infty  (1+r)^{-n-3}N_V(\lambda r)\>dr.
\end{equation}
If $V$ has support in a ball of radius $R$, then Borthwick \cite[Thm.~1.1]{Borthwick:2010}
gives an upper bound 
\[
N_V(r) \le C_R r^{n+1}.
\]
Applying this estimate to \eqref{cnl.split} gives
\[
N_V(\lambda a) \ge c_n\lambda^{n+1} - C_R \lambda^{n+1} a^{-1}.
\]
We can then set $a = 2C_R/c_n$ and rescale $\lambda$ to obtain
\[
N_V(\lambda) \ge \frac12 c_n \paren*{\frac{c_n}{2C_R}}^{\!n+1} \lambda^{n+1}.
\]
\end{proof}

The existence of a lower bound in even dimensions is not surprising,
since the optimal order of growth is already attained for $V=0$.    
It is more interesting to examine the difference between $\calR_V$ and the background resonance set.  
Note that when $n+1$ is even, the expansion of Theorem~\ref{sigmap.thm} 
contains only odd powers of $\xi$.  Since $\sigma'(\xi)$ is an even function, this creates a 
discrepancy that we can exploit.  

\begin{theorem}\label{even.scphase.thm}
For $n+1$ even, suppose that $V_1, V_2 \in \cinf_0(\Hn, \bbR)$.  If the resonance sets 
$\calR_{V_1}$ and $\calR_{V_2}$ differ by only finitely many points (counting multiplicities), then:
\begin{enumerate}
\item The corresponding scattering phases $\sigma_{V_1}$ and $\sigma_{V_2}$ differ by a constant.
\item The sets $\calR_{V_1} \backslash \calR_{V_2}$ and $\calR_{V_2} \backslash \calR_{V_1}$ are contained in $(0,n)$ 
and invariant under the reflection $s \mapsto n-s$.
\item The wave invariants satisfy $a_k(V_1) = a_k(V_2)$ for $k = 1, \dots, (n-1)/2$.
\end{enumerate}
Furthermore, if $\calR_{V_1} = \calR_{V_2}$ (with multiplicities), then $\Theta_{V_1} = \Theta_{V_2}$ and hence all of the wave invariants
match.
\end{theorem}

\begin{proof}
Under the assumption that $\calR_{V_1}$ and $\calR_{V_2}$ differ by only finitely many points, the factorization of 
Proposition~\ref{tau.factor.prop} implies that
\begin{equation}\label{tau.factor12}
\frac{\tau_{V_1}(s)}{\tau_{V_2}(s)} = (-1)^{m_{V_1}(\frac{n}2) - m_{V_2}(\frac{n}2)} e^{p(s)} 
\prod_{\zeta \in \calR_{V_1} \backslash \calR_{V_2}} \frac{n-s-\zeta}{s-\zeta} \prod_{\zeta \in \calR_{V_2} \backslash \calR_{V_1}} \frac{s-\zeta}{n-s-\zeta},
\end{equation}
where $p$ is a polynomial with degree at most $n+1$, satisfying $p(s) = p(n-s)$.  It follows that $\sigma_{V_1}'(\xi) - \sigma_{V_2}'(\xi)$ is an even, rational function of $\xi$.
Since the expansion formula from Theorem~\ref{sigmap.thm} contains only odd powers of $\xi$, plus a $O(\xi^{-\infty})$ remainder,
this implies that $\sigma_{V_1}'(\xi) = \sigma_{V_2}'(\xi)$.  The equality of the wave invariants for $k = 1,\dots,[n/2]$ 
then follows from the matching of expansion coefficients.  Since $\sigma_{V_1}' = \sigma_{V_2}'$ also implies that
$\tau_{V_1}(s)/\tau_{V_2}(s)$ is constant, the characterization of 
$\calR_{V_1} \backslash \calR_{V_2}$ and $\calR_{V_2} \backslash \calR_{V_1}$ follows from \eqref{tau.factor12}.

If $\calR_{V_1} = \calR_{V_2}$, then the same argument shows that the scattering phases are equal.  It then follows from
\eqref{Theta.vart} that $\Theta_{V_1} = \Theta_{V_1}$.  
\end{proof}

Let us apply Theorem~\ref{even.scphase.thm} to compare $\calR_V$ to $\calR_0$.  
The hypothesis that $\calR_V$ and $\calR_0$ differ by finitely many points implies that $\sigma'(\xi) = 0$ and $a_k(V) = 0$ for $k \le (n-1)/2$.
Since $\calR_0 \cap (0,n) = \emptyset$, it also implies that $\calR_V$ is the union of $\calR_0$ with a possible resonance at $\zeta = \tfrac{n}2$,
plus a finite set of pairs of the form
\[
\zeta = \tfrac{n}2 \pm \sqrt{\tfrac{n^2}4 - \lambda_j},
\]
where $\lambda_j$ is an eigenvalue.

As noted at the start of this section, the vanishing of the first two wave invariants implies that $V=0$.  From 
Theorem~\ref{even.scphase.thm} we thus immediately obtain the following:
\begin{corollary}
Let $V \in \cinf_0(\Hn, \bbR)$ with $n+1$ even and $n \ge 5$.  If $V \ne 0$ then $\calR_V$ differs from $\calR_0$
by infinitely many points (counting multiplicities). 
\end{corollary}

For $n\le 3$ we cannot fully control the first two wave invariants.  However, we
can derive some extra information from the heat trace.  
If $\sigma'(\xi)=0$, we see from \eqref{heat.bk} and \eqref{heat.trace.PV} that
\begin{equation}\label{rfinite.match}
\sum_{j=1}^d e^{t(\frac{n^2}4 - \lambda_j)}  + \tfrac12 m_V(\tfrac{n}2) \sim  \sum_{k=(n+1)/2}^\infty 2^{-n+2k-1} \pi^{-\frac12} a_k(V) t^{-\frac{n+1}2+k}
\end{equation}
as $t \to 0$.
Matching the coefficients in the expansion leads to a set of relationships between the discrete eigenvalues $\lambda_j$,
the multiplicity $m_V(\tfrac{n}2)$, and the wave invariants.

\begin{corollary}
For $V \in \cinf_0(\bbH^2, \bbR)$, if $V \ne 0$ and  $\int V\>dg \ge 0$, then $\calR_V$ differs from $\calR_0$ by infinitely many points.  
The same conclusion holds for $V \in \cinf_0(\bbH^4, \bbR)$, provided $\int V\>dg \ne 0$.
\end{corollary}

\begin{proof}
Assume that $\calR_V$ differs from $\calR_0$ by finitely many points.  For $n=1$ the $t^0$ term in \eqref{rfinite.match}
gives
\[
d +  \tfrac12 m_V(\tfrac{n}2) =  - \frac{1}{4\pi} \int_{\bbH^2} V(z)\>dg(z).
\]
Hence $\int V\>dg \ge 0$ implies $\calR_V = \emptyset$, which gives $V=0$ by Theorem~\ref{even.R0.thm}.

For $n=3$, the assumption that $\calR_V$ differs from $\calR_0$ by finitely many points gives $a_1(V) = 0$
by Theorem~\ref{even.scphase.thm}.  This means $\int V\>dg = 0$.
\end{proof}

Assuming a finite discrepancy between $\calR_V$ and $\calR_0$,
the expansion \eqref{rfinite.match} also implies a set of relations 
between eigenvalues and wave invariants.  For $n=1$, we have
\[
\frac{1}{(k-1)!} \sum_{j=1}^d \paren*{\tfrac{1}4 - \lambda_j}^{k-1} = 4^{k-1} \pi^{-\frac12} a_k(V)
\]
for $k \ge 2$, and if $n=3$,
\[
d +  \tfrac12 m_V(\tfrac{3}2) = \pi^{-\frac12} a_2(V)
\]
and
\[
\frac{1}{(k-2)!} \sum_{j=1}^d \paren*{\tfrac{9}4 - \lambda_j}^{k-2} = 4^{k-2} \pi^{-\frac12} a_k(V)
\]
for $k \ge 3$.
Although these relations seem rather delicate, they do not lead to any obvious contradiction.

\subsection{Odd dimensions}

In odd dimensions, the primary limitation to drawing implications from the wave trace is the fact that the
terms in the expansion of Theorem~\ref{wave.exp.thm} with $k \le n/2$ are distributions supported only at $t=0$.
Hence the trace formula of Theorem~\ref{poisson.thm} yields no information about the first
$n/2$ wave invariants.

In the Euclidean case, S\'a Barreto-Zworski \cite{SaZw} exploited the decay of the heat trace as 
$t \to \infty$ to prove an existence result.  In the hyperbolic case, 
the corresponding decay rate from Corollary~\ref{ht.decay.cor}, is merely $O(t^{-1/2})$, independent of the dimension.  
Hence this approach fails and we obtain an existence result only for dimension three.  

\begin{theorem}\label{odd.exist.thm}
For $V \in \cinf_0(\bbH^3, \bbR)$, if $V \ne 0$ then $\calR_V$ is not empty.  
\end{theorem}

\begin{proof}
For $n=2$, if $\calR_V = \emptyset$ then Theorems~\ref{poisson.thm}  and \ref{wave.exp.thm} show that $a_k(V) = 0$ for $k \ge 2$.  
By the formula from Proposition~\ref{d12.prop}, 
\[
a_2(V) =  \frac{1}{32\pi} \int_{\bbH^3} V(z)^2\>dg(z),
\]
and so $a_2(V) = 0$ implies $V=0$ when $n=2$.
\end{proof}

As long as at least one resonance exists, we can use the Poisson formula to show
that there are infinitely many.  The arguments from Christiansen \cite[Thm.~1]{Christ:1999b}
and S\'a Barreto \cite[Thm.~1.3]{SB:2001}
can then be applied to produce a  lower bound on the count.
\begin{theorem}
For $V \in \cinf_0(\Hn, \bbR)$ with $n+1$ odd, either $\calR_V = \emptyset$  or $\calR_V$ is infinite
and the counting function satisfies
\[
\limsup_{r \to \infty} \frac{N_V(r)}{r} > 0.
\]
\end{theorem}

\begin{proof}
Suppose that $\calR_V$ is finite.  By Theorem~\ref{poisson.thm} the wave
trace is given by a finite sum,
\[
\Theta_V(t) = \frac12 \sum_{\zeta\in\calR_V} e^{(\zeta-\frac{n}2)\abs{t}},
\]
for $t \ne 0$.  Hence
\[
\lim_{t \to 0} \Theta_V(t) = \frac12 \#\calR_V.
\]
Since the wave trace expansion of Theorem~\ref{wave.exp.thm} has no term of order $t^0$
for $n+1$ odd, this shows that $\calR_V = \emptyset$.

Now assume that $\calR_V$ is infinite.  Since $\calR_0 = \emptyset$ in odd dimensions, the factorization formula of 
Proposition~\ref{tau.factor.prop} reduces to 
\[
\tau(s) = (-1)^{m_V(\tfrac{n}2)} e^{q(s)} \frac{H_V(n-s)}{H_V(s)}.
\]
This is completely analogous to the factorization in the Euclidean case, once we shift the spectral parameter by setting 
$s = \nh + i\xi$.

Suppose that $N_V(r) = O(r)$.  Then, the scattering phase expansion 
of Corollary~\ref{sigma.cor} allows us to apply \cite[Thm.~1.2]{SB:2001} to deduce that
\[
\abs*{\sum_{\abs{\xi_j} < r } \frac{1}{\xi_j} }\le C,
\]
for all $r>0$. The argument from the proof of \cite[Thm.~1.3]{SB:2001} then yields 
a contradiction to the fact the asymptotic expansion from Corollary~\ref{sigma.cor} has only integral powers of $\xi$.
\end{proof}



\end{document}